\documentclass{scrartcl}

\usepackage{amsmath}
\usepackage{amsfonts}
\usepackage{amssymb}
\usepackage{pgf}
\usepackage{hyperref}

\newcommand{\bbbn}{\mathbb{N}}
\newcommand{\bbbr}{\mathbb{R}}

\newcommand{\Idx}{\mathcal{I}}
\newcommand{\Jdx}{\mathcal{J}}
\newcommand{\Kdx}{\mathcal{K}}

\newcommand{\ctI}{\mathcal{T}_{\Idx}}
\newcommand{\ctJ}{\mathcal{T}_{\Jdx}}
\newcommand{\ctK}{\mathcal{T}_{\Kdx}}
\newcommand{\ctIJ}{\mathcal{T}_{\Idx\times\Jdx}}
\newcommand{\ctJK}{\mathcal{T}_{\Jdx\times\Kdx}}
\newcommand{\ctIK}{\mathcal{T}_{\Idx\times\Kdx}}

\newcommand{\lfI}{\mathcal{L}_{\Idx}}

\newcommand{\lfIJ}{\mathcal{L}_{\Idx\times\Jdx}}
\newcommand{\lfJK}{\mathcal{L}_{\Jdx\times\Kdx}}

\newcommand{\lfaIJ}{\mathcal{L}^+_{\Idx\times\Jdx}}
\newcommand{\lfaIK}{\mathcal{L}^+_{\Idx\times\Kdx}}
\newcommand{\lfiIJ}{\mathcal{L}^-_{\Idx\times\Jdx}}
\newcommand{\lfaJK}{\mathcal{L}^+_{\Jdx\times\Kdx}}

\newcommand{\ctprod}{\mathcal{T}_{\Idx\times\Kdx}^{XY}}
\newcommand{\lfprod}{\mathcal{L}^{XY}_{\Idx\times\Kdx}}
\newcommand{\lfaprod}{\mathcal{L}^{XY,+}_{\Idx\times\Kdx}}
\newcommand{\lfiprod}{\mathcal{L}^{XY,-}_{\Idx\times\Kdx}}

\newcommand{\treeroot}{\mathop{\operatorname{root}}\nolimits}
\newcommand{\chil}{\mathop{\operatorname{chil}}\nolimits}

\newcommand{\supp}{\mathop{\operatorname{supp}}\nolimits}
\newcommand{\brow}{\mathop{\operatorname{row}}\nolimits}

\newcommand{\adm}{\mathop{\operatorname{adm}}\nolimits}

\newtheorem{theorem}{Theorem}
\newtheorem{definition}[theorem]{Definition}
\newtheorem{lemma}[theorem]{Lemma}
\newtheorem{example}[theorem]{Example}
\newtheorem{remark}[theorem]{Remark}

\newenvironment{proof}{\emph{Proof.}}{\hfill $\Box$}

\begin{document}
\title{Adaptive fast multiplication of $\mathcal{H}^2$-matrices}
\author{Steffen B\"orm\thanks{Mathematisches Seminar, Universit\"at Kiel,
    Heinrich-Hecht-Platz 6, 24118 Kiel, Germany\hfill\strut\linebreak boerm@math.uni-kiel.de}}

\maketitle

\begin{abstract}
Hierarchical matrices approximate a given matrix by a decomposition
into low-rank submatrices that can be handled efficiently in factorized
form.
$\mathcal{H}^2$-matrices refine this representation following the ideas
of fast multipole methods in order to achieve linear, i.e., optimal
complexity for a variety of important algorithms.

The matrix multiplication, a key component of many more advanced
numerical algorithms, has so far proven tricky:
the only linear-time algorithms known so far either require the very
special structure of HSS-matrices or need to know a suitable basis
for all submatrices in advance.

In this article, a new and fairly general algorithm for multiplying
$\mathcal{H}^2$-matrices in linear complexity with adaptively constructed
bases is presented.
The algorithm consists of two phases: first an intermediate representation
with a generalized block structure is constructed, then this representation
is re-compressed in order to match the structure prescribed by the
application.

The complexity and accuracy are analyzed and numerical experiments
indicate that the new algorithm can indeed be significantly faster
than previous attempts.
\end{abstract}

% ============================================================
% Introduction
% ============================================================
\section{Introduction}

The matrix multiplication, in the general form $Z \gets Z + X Y$ with
matrices $X$, $Y$, and $Z$, is a central operation of linear algebra
and can be used, e.g., to express inverses, LU and Cholesky
factorizations, orthogonal decompositions, and matrix functions.
While the standard definition leads to a complexity of $\mathcal{O}(n^3)$
for $n$-dimensional matrices, Strassen's famous algorithm \cite{ST69}
reduces the complexity to $\mathcal{O}(n^{\log_2(7)})$, and the exponent
has since been significantly reduced by further work.
Since the resulting $n\times n$ matrix has $n^2$ coefficients, it seems
clear that no algorithm can compute the product of two $n\times n$
matrices explicitly in less than $n^2$ operations.

The method of \emph{hierarchical matrices} \cite{HA99,GRHA02} introduced by
Hackbusch restricts to structured matrices, i.e., by assuming that matrices
can be represented by mosaics of low-rank matrices, and shows that this
can significantly reduce the complexity for the multiplication and inversion.
It has been proven \cite{BOHA02a,BEHA03,GAHAKH00,GAHAKH02,GR01a,GRHAKH02}
that matrices appearing in many important applications, ranging from
elliptic partial differential equations to solutions of matrix equations,
can be approximated accurately in this way, and efficient algorithms have
been developed to compute these approximations efficiently.

Hierarchical matrix algorithms typically require
$\mathcal{O}(n \log n)$ or $\mathcal{O}(n \log^2 n)$ operations.
Combining these techniques with concepts underlying the \emph{fast
multipole method} \cite{RO85,GRRO87,GRRO97} leads to
\emph{$\mathcal{H}^2$-matrices \cite{HAKHSA00,BOHA02,BO10} that
represent low-rank blocks using fixed nested \emph{cluster bases}
and require} only linear complexity for a number of important
algorithms, i.e., they can reach the optimal order of complexity.

While optimal-order algorithms for a number of $\mathcal{H}^2$-matrix
operations have been developed \cite{BO10}, the question of efficient
arithmetic algorithms has been open for decades.
If the block structure is severely restricted, leading to
\emph{hierarchically semi-separable matrices}, the matrix multiplication
and certain factorizations can be implemented in linear complexity
\cite{CHGULY05,CHGUPA06,CHGULIXI09b}, but the corresponding block
structure is not suitable for general applications.
For matrices resulting from one-dimensional fast multipole methods,
a multiplication algorithm has been presented in \cite{LE20}.

For the LU factorization of $\mathcal{H}^2$-matrices, \cite{MAJI18}
offers a promising algorithm, but its treatment of the fill-in appearing
during the elimination process may either lead to reduced accuracy
or increased ranks, and therefore less-than-optimal complexity.
An alternative approach focuses on the efficient parallelization of the
$\mathcal{H}^2$-matrix LU factorization and introduces an interesting
technique for avoiding certain negative effects of fill-in to obtain
impressive experimental results for large matrices \cite{MADEYO22}.

If the cluster bases used to represent the submatrices
of the result of an $\mathcal{H}^2$-matrix multiplication are given a priori,
the matrix multiplication can be performed in linear complexity for
fairly general block structures \cite{BO04a}, but for important
applications like solution operators for partial differential equations,
no practically useful bases are known in general, and
finding these bases by an efficient algorithm in a way that guarantees
a sufficient accuracy is a major challenge.

A key observation underlying all efficient algorithms for
rank-structured matrices is that on the one hand we have to take advantage
of the rank structure, while on the other hand we have to ensure that it
is preserved across all the steps of an algorithm.
The most flexible algorithm developed so far \cite{BORE14} makes use
of the fact that products of $\mathcal{H}^2$-matrices locally have
a \emph{semi-uniform} rank structure that allows us to represent
these products efficiently, and an update procedure can be used
to accumulate the local contributions and approximate the final
result in $\mathcal{H}^2$-matrix form.
Since the current update procedure is fairly general, it reaches
a sub-optimal complexity of only $\mathcal{O}(n \log n)$.

This article presents an entirely new approach to approximating
the product of two $\mathcal{H}^2$-matrices:
it has been known for some time that the product can be represented
\emph{exactly} as an $\mathcal{H}^2$-matrix with a refined block
structure, cf. Section~\ref{se:coarsening}, and
``induced'' cluster bases of significantly
larger local ranks \cite[Chapter~7.8]{BO10}, cf. Section~\ref{se:induced_bases},
but this representation is computationally far too expensive to be
of practical use, although of linear complexity.
In order to obtain a useful approximation of the
product, this intermediate representation has to be re-compressed.
In this paper, two algorithms are presented that together
perform this task.
The first algorithm constructs a low-rank approximation of the product
in a refined block structure arising implicitly during the computation, and
the second algorithm transforms this intermediate block structure into a
coarser block structure prescribed by the user while preserving the given
accuracy.
Splitting the computation into two algorithms is motivated
by the fact that the complexity of the second algorithm depends
\emph{cubically} on the rank, so applying it directly to the uncompressed
induced cluster bases would be far too time-consuming.
The first algorithm significantly reduces the ranks of the cluster bases
without hurting the approximation quality and thereby allows the second
algorithm to work efficiently.

For the simpler structure of hierarchical matrices,
a similar approach has been investigated in \cite{DOHAMU19}:
the exact product is represented by lists of low-rank updates that
are then compressed locally.
For $\mathcal{H}^2$-matrices, the intermediate representation of the
exact result can be made far more efficient than for hierarchical
matrices, but the compression into a final $\mathcal{H}^2$-matrix
is considerably more challenging.

We will see that both algorithms have linear complexity under certain
conditions, so that the entire $\mathcal{H}^2$-matrix multiplication
can be performed in linear complexity, too.
For the first algorithm, a locally adaptive error control strategy
is presented, while for the second algorithm, the relative error
can be controlled reliably for every single block of the resulting
$\mathcal{H}^2$-matrix.
The complexity analysis relies implicitly on the
assumption that the products can be represented by $\mathcal{H}^2$-matrices
with low ranks.

Compared to sampling algorithms \cite{LILUYI11,HATO24,LEMA22},
the compression algorithm presented here can guarantee a given accuracy
and does not rely on special structures like hierarchically semi-separable
matrices.

Numerical experiments show that the entire procedure is signficantly
faster than previous algorithms for hierarchical matrices and indicate
that indeed linear complexity is reached in practice, even for fairly
complicated block structures appearing in the context of boundary element
methods on two-dimensional surfaces in three-dimensional space.

% ============================================================
% H^2-matrices
% ============================================================
\section{\texorpdfstring{$\mathcal{H}^2$-matrices}{H2-matrices}}

Our algorithm represents densely populated matrices, e.g., corresponding
to solution operators of partial differential equations or boundary integral
operators, by local low-rank approximations.
In the context of many-particle systems, this approach is known as the
\emph{fast multipole method} \cite{RO85,GRRO87,AN92,GRRO97,GIRO02},
its algebraic counterpart are \emph{$\mathcal{H}^2$-matrices}
\cite{HAKHSA00,BOHA02,BO10}, a refined version of the far more general
\emph{hierarchical matrices} \cite{HA99,HA15}.

To introduce $\mathcal{H}^2$-matrices, we follow
the approach of \cite{BOHA02,BO10}.
Let us consider a matrix $G\in\bbbr^{\Idx\times\Jdx}$ with a finite row
index set $\Idx$ and a finite column index set $\Jdx$.
We look for submatrices $G|_{\hat t\times\hat s}$ with
$\hat t\subseteq\Idx$, $\hat s\subseteq\Jdx$ that can be approximated
efficiently.

In order to find these submatrices efficiently, we split the
index sets $\Idx$ and $\Jdx$ hierarchically into \emph{clusters}
organized in trees.
When dealing with integral or partial differential equations, the
clusters will usually correspond to a hierarchical decomposition of
the computational domain.

We denote trees by the symbol $\mathcal{T}$, their nodes by
$t\in\mathcal{T}$, the root by $\treeroot(\mathcal{T})$, and the
children of any node $t\in\mathcal{T}$ by $\chil(t)\subseteq\mathcal{T}$.
We use \emph{labeled} trees, i.e., we associate every node
$t\in\mathcal{T}$ with a label $\hat t$ chosen from a suitable set.

%
% Definition: Cluster tree
%
\begin{definition}[Cluster tree]
Let $\mathcal{T}$ be a tree with root $r:=\treeroot(\mathcal{T})$.
We call $\mathcal{T}$ a \emph{cluster tree} for the index set $\Idx$
if
\begin{itemize}
  \item the root is labeled with the entire index set $\hat r = \Idx$,
  \item the union of the childrens' labels is the parent's label, i.e.,
    \begin{align*}
      \hat t &= \bigcup_{t'\in\chil(t)} \hat t' &
      &\text{ for all } t\in\mathcal{T} \text{ with } \chil(t)\neq\emptyset,
    \end{align*}
  \item siblings are disjoint, i.e.,
    \begin{align*}
      t_1\neq t_2 &\Rightarrow \hat t_1\cap\hat t_2 = \emptyset &
      &\text{ for all } t\in\mathcal{T},\ t_1,t_2\in\chil(t).
    \end{align*}
\end{itemize}
A cluster tree for $\Idx$ is usually denoted by $\ctI$, its leaves
by $\lfI := \{ t\in\ctI\ :\ \chil(t)=\emptyset \}$, and its nodes
are usually called \emph{clusters}.
\end{definition}

Cluster trees for various applications can be constructed
efficiently by a variety of algorithms \cite{HA15,GRKRLE05a}.
Their key property is that they provide us with a systematic way
to split index sets into subsets, allowing us to find submatrices
that can be approximated efficiently.
These submatrices are best also kept in a tree structure to allow
algorithms to quickly navigate through the structures of matrices.

%
% Definition: Block tree
%
\begin{definition}[Block tree]
Let $\mathcal{T}$ be a tree with root $r := \treeroot(\mathcal{T})$.
We call $\mathcal{T}$ a \emph{block tree} for two cluster trees
$\ctI$ and $\ctJ$ if
\begin{itemize}
  \item for every $b\in\mathcal{T}$ there are $t\in\ctI$ and
    $s\in\ctJ$ with $b=(t,s)$ and $\hat b = \hat t\times\hat s$,
  \item the root is the pair of the roots of the
    cluster trees, i.e., $r = (\treeroot(\ctI),\treeroot(\ctJ))$,
  \item if $b=(t,s)\in\mathcal{T}$ has children, they are
    \begin{align*}
      \chil(b) &= \begin{cases}
        \chil(t)\times\chil(s) &\text{ if } \chil(t)\neq\emptyset,
                                            \chil(s)\neq\emptyset,\\
        \chil(t)\times\{s\} &\text{ if } \chil(t)\neq\emptyset,
                                         \chil(s)=\emptyset,\\
        \{t\}\times\chil(s) &\text{ if } \chil(t)=\emptyset,
                                         \chil(s)\neq\emptyset.
      \end{cases}
    \end{align*}
\end{itemize}
A block tree for $\ctI$ and $\ctJ$ is usually denoted by
$\ctIJ$, its leaves by $\lfIJ$, and its nodes are usually called
\emph{blocks}.
\end{definition}

We can see that this definition implies
\begin{align*}
  \hat b &= \bigcup_{b'\in\chil(b)} \hat b' &
  &\text{ for all } b\in\ctIJ \text{ with } \chil(b)\neq\emptyset,
\end{align*}
and these unions are disjoint due to the properties of the
cluster trees $\ctI$ and $\ctJ$.
This means that the leaves $\lfIJ$ of a block tree describe
a disjoint partition of $\Idx\times\Jdx$, i.e., a decomposition
of the matrix $G$ into submatrices $G|_{\hat b} = G|_{\hat t\times\hat s}$,
$b=(t,s)\in\lfIJ$.

The central observation of hierarchical matrix methods is that a large
number of these submatrices can be approximated by low-rank matrices.
We call a block $(t,s)$ \emph{admissible} if we expect
to be able to approximate it by a low-rank matrix, and we stop subdividing
blocks in the construction of the block tree as soon as a block is
admissible.
This implies that all non-leaf blocks are inadmissible and the leaf blocks
$\lfIJ$ can be split into admissible leaves $\lfaIJ\subseteq\lfIJ$ and
a remainder $\lfiIJ := \lfIJ\setminus\lfaIJ$ of blocks that are
inadmissible and cannot be split further.
For a hierarchical matrix, we simply assume that every submatrix
$G|_{\hat b}$ for an admissible $b\in\lfaIJ$ can be approximated
by a low-rank matrix.

In order to obtain the higher efficiency of $\mathcal{H}^2$-matrices,
we introduce an additional restriction:
we only admit
\begin{align}\label{eq:vsw}
  G|_{\hat t\times\hat s}
  &= V_t S_{ts} W_s^* &
  &\text{ for all } b=(t,s)\in\lfaIJ,
\end{align}
where $S_{ts}\in\bbbr^{k\times k}$ is a \emph{coupling matrix} with a
small rank $k\in\bbbn$ and $V_t$ and $W_s$ are \emph{cluster bases}
with a specific hierarchical structure.

%
% Definition: Cluster basis
%
\begin{definition}[Cluster basis]
Let $\ctI$ be a cluster tree and $k\in\bbbn$.
A family $(V_t)_{t\in\ctI}$ of matrices satisfying
\begin{align*}
  V_t &\in \bbbr^{\hat t\times k} &
  &\text{ for all } t\in\ctI
\end{align*}
is called a \emph{cluster basis} of rank $k$ for the cluster tree $\ctI$
if for every $t\in\ctI$ and $t'\in\chil(t)$ there is a
\emph{transfer matrix} $E_{t'}\in\bbbr^{k\times k}$ with
\begin{equation}\label{eq:transfer}
  V_t|_{\hat t'\times k} = V_{t'} E_{t'},
\end{equation}
i.e., if the basis for the parent cluster can be expressed in terms of
the bases for its children.
\end{definition}

%
% Definition: H^2-matrix
%
\begin{definition}[$\mathcal{H}^2$-matrix]
\label{de:h2matrix}
Let $V=(V_t)_{t\in\ctI}$ and $W=(W_s)_{s\in\ctJ}$ be cluster bases for
$\ctI$ and $\ctJ$.
If (\ref{eq:vsw}) holds for all $b=(t,s)\in\lfaIJ$, we call $G$
an \emph{$\mathcal{H}^2$-matrix} with \emph{row basis} $V$ and
\emph{column basis} $W$.
The matrices $S_{ts}$ are called \emph{coupling matrices}.
\end{definition}

The submatrices $G|_{\hat t\times\hat s}$ for inadmissible leaves
$(t,s)\in\lfiIJ$ are stored without compression, since we may assume
that these matrices are small.

%
% Example: Integral operator
%
\begin{example}[Integral operator]
The discretization of integral operators can serve as a motivating
example for the structure of $\mathcal{H}^2$-matrices.
Let $\Omega\subseteq\bbbr^d$ be a domain or manifold, let
$g\colon\Omega\times\Omega\to\bbbr$ be a function.
For a Galerkin discretization, we choose families $(\varphi_i)_{i\in\Idx}$
and $(\psi_j)_{j\in\Jdx}$ of basis functions on $\Omega$ and introduce
the stiffness matrix $G\in\bbbr^{\Idx\times\Jdx}$ by
\begin{align}\label{eq:G_matrix}
  g_{ij} &= \int_\Omega \varphi_i(x) \int_\Omega g(x,y) \psi_j(y) \,dy \,dx &
  &\text{ for all } i\in\Idx,\ j\in\Jdx.
\end{align}
If the function $g$ is sufficiently smooth on a subdomain
$t\times s$ with $t,s\subseteq\Omega$, we can approximate it by
interpolation:
\begin{align*}
  g(x,y) &\approx \sum_{\nu=1}^k \sum_{\mu=1}^k
             g(\xi_{t,\nu},\xi_{s,\mu}) \ell_{t,\nu}(x) \ell_{s,\mu}(y) &
  &\text{ for all } x\in t,\ y\in s,
\end{align*}
where $(\xi_{t,\nu})_{\nu=1}^k$ and $(\xi_{s,\mu})_{\mu=1}^k$ are the
interpolation points in the subdomains $t$ and $s$, respectively,
and $(\ell_{t,\nu})_{\nu=1}^k$ and $(\ell_{s,\mu})_{\mu=1}^k$ are the
corresponding Lagrange polynomials.

Substituting the interpolating polynomial for $g$ in (\ref{eq:G_matrix})
yields
\begin{align*}
  g_{ij} &\approx \sum_{\nu=1}^k \sum_{\mu=1}^k
           \underbrace{\int_\Omega \varphi_i(x) \ell_{t,\nu}(x) \,dx
                      }_{=:v_{t,i\nu}}
           \underbrace{g(\xi_{t,\nu},\xi_{s,\mu})
                      }_{=:s_{ts,\nu\mu}}
           \underbrace{\int_\Omega \psi_j(y) \ell_{s,\mu}(y) \,dy
                      }_{=:w_{s,j\mu}} &
  &\text{ for all } i\in\hat t,\ j\in\hat s,
\end{align*}
where $\hat t := \{ i\in\Idx\ :\ \supp\varphi_i\subseteq t \}$ and
$\hat s := \{ j\in\Jdx\ :\ \supp\psi_j\subseteq s \}$ are the indices
supported in $t$ and $s$, respectively.
This is precisely the representation (\ref{eq:vsw}) required by the
definition of $\mathcal{H}^2$-matrices.

If we use the same order of interpolation for all clusters,
the identity theorem for polynomials yields
\begin{align*}
  \ell_{t,\nu} &= \sum_{\nu'=1}^k
       \underbrace{\ell_{t,\nu}(\xi_{t',\nu'})}_{=:e_{t',\nu'\nu}}
       \ell_{t',\nu'} &
  &\text{ for all } \nu\in\{1,\ldots,k\},
\end{align*}
and this gives rise to (\ref{eq:transfer}).
Interpolation is just one way to obtain $\mathcal{H}^2$-matrices
\cite{BOLOME02,BOSA03}, common alternatives are Taylor expansion
\cite{HAKHSA00} and multipole expansions \cite{RO85,GRRO97}.
\end{example}

An $\mathcal{H}^2$-matrix requires only $\mathcal{O}(n k)$ units
of storage, and the matrix-vector multiplication can be performed
in $\mathcal{O}(n k)$ operations \cite{HAKHSA00,BO10}.
These properties make $\mathcal{H}^2$-matrices very attractive for
handling large dense matrices.

% ============================================================
% Matrix multiplication and induced bases
% ============================================================
\section{Matrix multiplication and induced bases}
\label{se:induced_bases}

In the following we consider the multiplication of $\mathcal{H}^2$-matrices,
i.e., we assume that $\mathcal{H}^2$-matrices $X\in\bbbr^{\Idx\times\Jdx}$
and $Y\in\bbbr^{\Jdx\times\Kdx}$ with coupling matrices
$(S_{X,ts})_{(t,s)\in\lfaIJ}$ and $(S_{Y,sr})_{(s,r)\in\lfaJK}$ are given
and we are looking to approximate the product $X Y$ efficiently.

We denote the row and column cluster bases of $X$ by
$V_X=(V_{X,t})_{t\in\ctI}$ and $W_X=(W_{X,s})_{s\in\ctJ}$ with
corresponding transfer matrices $(E_{X,t'})_{t'\in\ctI}$ and $(F_{X,s'})_{s'\in\ctJ}$.

The row and column cluster bases of $Y$ are denoted by
$V_Y=(V_{Y,s})_{s\in\ctJ}$ and $W_Y=(W_{Y,r})_{r\in\ctK}$ with
corresponding transfer matrices $(E_{Y,s'})_{s'\in\ctJ}$ and $(F_{Y,r'})_{r'\in\ctK}$.

We will use a recursive approach to compute the product $X Y$ that
considers products of submatrices
\begin{equation}\label{eq:XY}
  X|_{\hat t\times\hat s} Y|_{\hat s\times\hat r}
\end{equation}
with $t\in\ctI$, $s\in\ctJ$, and $r\in\ctK$.
Depending on the relationships between $t$, $s$, and $r$, these
products have to be handled differently.

If $(t,s)$ and $(s,r)$ are elements of the block trees $\ctIJ$ and $\ctJK$,
but not leaves, we can switch to their children $t'\in\chil(t)$,
$s'\in\chil(s)$, $r'\in\chil(r)$ and construct the product (\ref{eq:XY})
from the sub-products
\begin{equation*}
  X|_{\hat t'\times\hat s'} Y|_{\hat s'\times\hat r'}.
\end{equation*}
This leads to a simple recursive algorithm.

The situation changes if $(t,s)$ or $(s,r)$ are leaves.
Assume that $(s,r)\in\lfaJK$ is an admissible leaf.
Our Definition~\ref{de:h2matrix} yields
\begin{equation*}
  Y|_{\hat s\times\hat r} = V_{Y,s} S_{Y,sr} W_{Y,r}^*,
\end{equation*}
and the product (\ref{eq:XY}) takes the form
\begin{equation}\label{eq:semiuniform}
  X|_{\hat t\times\hat s} Y|_{\hat s\times\hat r}
  = \underbrace{X|_{\hat t\times\hat s} V_{Y,s}}_{A_{ts}} S_{Y,sr} W_{Y,r}^*
  = A_{ts} S_{Y,sr} W_{Y,r}^*.
\end{equation}
This equation reminds us of the defining equation (\ref{eq:vsw}) of
$\mathcal{H}^2$-matrices, but the indices do not match.
Formally, we can fix this by collecting \emph{all} matrices $A_{ts}$
with $b=(t,s)\in\lfiIJ$ in a new cluster basis for the cluster $t$.
If we include $V_{X,t}$, we can also cover the case
$b=(t,s)\in\lfaIJ$ using (\ref{eq:vsw}).

%
% Definition: Induced row cluster basis
%
\begin{definition}[Induced row cluster basis]
\label{de:induced_row}
For every $t\in\ctI$ let
\begin{equation*}
  \mathcal{C}_t := \{ s\in\ctJ\ :\ (t,s)\in\ctIJ \text{ inadmissible} \}
\end{equation*}
denote the set of column clusters of inadmissible blocks with the row
cluster $t$.
Combining all matrices $A_{ts}$ for all $s\in\mathcal{C}_t$ from
(\ref{eq:semiuniform}) yields
\begin{equation*}
  V_t := \begin{pmatrix}
            V_{X,t} & X|_{\hat t\times\hat s_1} V_{Y,s_1} &
            \ldots & X|_{\hat t\times\hat s_m} V_{Y,s_m}
         \end{pmatrix}
\end{equation*}
where we have enumerated $\mathcal{C}_t = \{s_1,\ldots,s_m\}$ for the
sake of simplicity.

This is again a cluster basis \cite[Chapter~7.8]{BO10}, although the
associated ranks are generally far higher than for $V_X$ or $V_Y$.
\end{definition}

By padding $S_{Y,sr}$ in (\ref{eq:semiuniform}) with rows of zeros
to obtain a matrix $S_{tr}$ such that $X|_{\hat t\times\hat s} V_{Y,s}
S_{Y,sr} = V_t S_{tr}$, we find
\begin{equation*}
  X|_{\hat t\times\hat s} Y|_{\hat s\times\hat r} = V_t S_{tr} W_{Y,r}^*,
\end{equation*}
i.e., we have constructed a factorized representation as in (\ref{eq:vsw}).

We can use a similar approach if $(t,s)\in\lfaIJ$ is an admissible
leaf.
Definition~\ref{de:h2matrix} gives us the factorized representation
\begin{gather}
  X|_{\hat t\times\hat s} = V_{X,t} S_{X,ts} W_{X,s}^*,\notag\\
  X|_{\hat t\times\hat s} Y|_{\hat s\times\hat r}
  = V_{X,t} S_{X,ts} W_{X,s}^* Y|_{\hat s\times\hat r}
  = V_{X,t} S_{X,ts}
    (\underbrace{Y|_{\hat s\times\hat r}^*  W_{X,s}}_{=:B_{sr}})^*
  = V_{X,t} S_{X,ts} B_{sr}^*.\label{eq:semiuniform2}
\end{gather}
Now we collect all matrices $B_{sr}$ in a new column cluster basis for
the cluster $r\in\ctK$ and also include $W_{Y,r}$ to cover $b=(s,r)\in\lfaJK$.

%
% Definition: Induced row cluster basis
%
\begin{definition}[Induced column cluster basis]
\label{de:induced_col}
For every $r\in\ctK$ let
\begin{equation*}
  \mathcal{C}_r := \{ s\in\ctJ\ :\ (s,r)\in\ctJK \text{ inadmissible} \}
\end{equation*}
denote the set of row clusters of inadmissible blocks with the column
cluster $r$.
Combining all matrices $B_{sr}$ for all $s\in\mathcal{C}_r$ from
(\ref{eq:semiuniform2}) yields
\begin{equation*}
  W_r := \begin{pmatrix}
            W_{Y,r} & Y|_{\hat s_1\times\hat r}^* W_{X,s_1} &
            \ldots & Y|_{\hat s_m\times\hat r}^* W_{X,s_m}
         \end{pmatrix}
\end{equation*}
where we have enumerated $\mathcal{C}_r = \{s_1,\ldots,s_m\}$ for the
sake of simplicity.

This is again a cluster basis \cite[Chapter~7.8]{BO10}, although the
associated ranks are generally far higher than for $W_X$ or $W_Y$.
\end{definition}

As before, we can construct $S_{tr}$ by padding $S_{X,ts}$
in (\ref{eq:semiuniform2}) with columns of zeros to get
$S_{X,ts} W_{X,s}^* Y|_{\hat s\times\hat r} = S_{tr} W_r^*$ and obtain
\begin{equation*}
  X|_{\hat t\times\hat s} Y|_{\hat s\times\hat r}
  = V_{X,t} S_{tr} W_r^*,
\end{equation*}
i.e., a low-rank factorized representation of the product.

Using the induced row and column bases, we can represent the product
$XY$ \emph{exactly} as an $\mathcal{H}^2$-matrix by simply splitting
products $X|_{\hat t\times\hat s} Y|_{\hat s\times\hat r}$ until
$(t,s)$ or $(s,r)$ are admissible and we can represent the product
in the new bases.
If we reach leaves of the cluster tree, we can afford to store the product
directly without looking for a factorization.

Unless we are restricting our attention to very simple block
structures \cite{HAKHKR04,CHGULY05}, this $\mathcal{H}^2$-matrix
representation of the product
\begin{itemize}
  \item requires ranks that are far too high and
  \item far too many blocks to be practically useful.
\end{itemize}
The goal for this article is to present algorithms for
efficiently constructing an approximation of the product using
optimized cluster bases and a prescribed block structure.
In the next section, we will investigate how the induced cluster
bases can be compressed.
The following section is then dedicated to coarsening the block
structure.

% ============================================================
% Compressed induced cluster bases
% ============================================================
\section{Compressed induced cluster bases}

We focus on the induced row cluster basis introduced in
Definition~\ref{de:induced_row}, since the induced column cluster
basis has a very similar structure and the adaptation of the
present algorithm is straightforward.

Our goal is to find a cluster basis $(Q_t)_{t\in\ctI}$ that can
approximate all products (\ref{eq:semiuniform}) sufficiently
well.
In order to avoid redundancies and to allow elegant error estimates,
we focus on \emph{isometric} cluster bases.

%
% Definition: Isometric cluster basis
%
\begin{definition}[Isometric cluster basis]
A cluster basis $(Q_t)_{t\in\ctI}$ is called \emph{isometric} if
\begin{align*}
  Q_t^* Q_t &= I &
  &\text{ for all } t\in\ctI.
\end{align*}
\end{definition}

For an isometric cluster basis, the best approximation of a
matrix in the range of this basis is given by the orthogonal
projection $Q_t Q_t^*$.
In our case, we want to approximate the products (\ref{eq:semiuniform}),
i.e., we require
\begin{align*}
  X|_{\hat t\times\hat s} Y|_{\hat s\times\hat r}
  &\approx Q_t Q_t^* X|_{\hat t\times\hat s} Y|_{\hat s\times\hat r} &
  &\text{ for all } s\in\mathcal{C}_t,\ (s,r)\in\lfaJK.
\end{align*}
Due to $(s,r)\in\lfaJK$, we have $Y|_{\hat s\times\hat r} = V_{Y,s}
S_{Y,sr} W_{Y,r}^*$ and therefore
\begin{align}\label{eq:original_product}
  X|_{\hat t\times\hat s} V_{Y,s} S_{Y,sr} W_{Y,r}^*
  &\approx Q_t Q_t^* X|_{\hat t\times\hat s} V_{Y,s} S_{Y,sr} W_{Y,r}^* &
  &\text{ for all } s\in\mathcal{C}_t,\ (s,r)\in\lfaJK.
\end{align}
Computing the original matrix $X|_{\hat t\times\hat s} V_{Y,s} S_{Y,sr}
W_{Y,r}^*$ directly would be far too computationally
expensive, since $W_{Y,r}$ has a large number of rows if $\hat r$ is
large.

This is where \emph{condensation} is a useful strategy:
Applying an orthogonal transformation to the equation from the right
does not change the approximation properties, and taking advantage of
the fact that $W_{Y,r}$ has only $k$ columns allows us to significantly
reduce the matrix dimension without any effect on the approximation
quality.

%
% Definition: Basis weight
%
\begin{definition}[Basis weight]
\label{de:basis_weights}
For every $r\in\ctK$, there are a matrix $R_{Y,r}\in\bbbr^{\ell\times k}$
and an isometric matrix $Q_{Y,r}\in\bbbr^{\hat r\times\ell}$ such that
$W_{Y,r} = Q_{Y,r} R_{Y,r}$ and $\ell=\min\{k,\#\hat r\}$.

The matrices $(R_{Y,r})_{r\in\ctK}$ are called the \emph{basis weights}
for the cluster basis $(W_{Y,r})_{r\in\ctK}$.
\end{definition}

%
% Figure: Construction of basis weights
%
\begin{figure}
  \begin{quotation}
    \begin{tabbing}
      \textbf{procedure} basis\_weights($r$);\\
      \textbf{begin}\\
      \quad\= \textbf{if} $\chil(r)=\emptyset$ \textbf{then}\\
      \> \quad\= Compute a thin Householder
                    factorization $Q_{Y,r} R_{Y,r} = W_{Y,r}$\\
      \> \textbf{else begin}\\
      \> \> \textbf{for} $r'\in\chil(r)$ \textbf{do}
               basis\_weights($r'$);\\
      \> \> $\widehat{W}_{Y,r} \gets \begin{pmatrix}
               R_{Y,r_1} F_{Y,r_1}^*\\
               \vdots\\
               R_{Y,r_c} F_{Y,r_c}^*
             \end{pmatrix}$ with $\chil(r)=\{r_1,\ldots,r_c\}$;\\
      \> \> Compute a thin Householder
              factorization $\widehat{Q}_{Y,r} R_{Y,r} = \widehat{W}_{Y,r}$\\
      \> \textbf{end}\\
      \textbf{end}
    \end{tabbing}
  \end{quotation}
  \caption{Construction of the basis weights for the basis $W_Y$}
  \label{fi:basis_weights}
\end{figure}

The basis weights can be efficiently computed in $\mathcal{O}(n k^2)$
operations by a recursive algorithm \cite[Chapter~5.4]{BO10},
cf. Figure~\ref{fi:basis_weights}:
in the leaves, we compute the thin Householder factorization
$W_{Y,r} = Q_{Y,r} R_{Y,r}$ directly.
If $r$ has children, let's say $\chil(r)=\{r_1',r_2'\}$, we assume that
the factorizations for the children have already been computed and observe
\begin{equation*}
  W_{Y,r}
  = \begin{pmatrix}
      W_{Y,r_1'} F_{Y,r_1'}\\
      W_{Y,r_2'} F_{Y,r_2'}
    \end{pmatrix}
  = \begin{pmatrix}
      Q_{Y,r_1'} R_{Y,r_1'} F_{Y,r_1'}\\
      Q_{Y,r_2'} R_{Y,r_2'} F_{Y,r_2'}
    \end{pmatrix}
  = \begin{pmatrix}
      Q_{Y,r_1'} & \\
      & Q_{Y,r_2'}
    \end{pmatrix}
    \underbrace{\begin{pmatrix}
      R_{Y,r_1'} F_{Y,r_1'}\\
      R_{Y,r_2'} F_{Y,r_2'}
    \end{pmatrix}}_{=:\widehat{W}_{Y,r}}.
\end{equation*}
Here $F_{Y,r_1'}$ and $F_{Y,r_2'}$ denote the transfer
matrices for the children $\chil(r)=\{r_1',r_2'\}$ and the matrix $W_{Y,r}$.
The left factor is already isometric, and with the thin Householder
factorization $\widehat{W}_{Y,r} = \widehat{Q}_{Y,r} R_{Y,r}$ of the
small matrix $\widehat{W}_{Y,r}$ we find
\begin{equation*}
  W_{Y,r}
  = \underbrace{\begin{pmatrix}
      Q_{Y,r_1'} & \\
      & Q_{Y,r_2'}
    \end{pmatrix} \widehat{Q}_{Y,r}}_{=:Q_{Y,r}} R_{Y,r} = Q_{Y,r} R_{Y,r}.
\end{equation*}
Since the matrices $Q_{Y,r}$ associated with the basis weight are
isometric, we have
\begin{equation*}
  \|(I - Q_t Q_t^*) X|_{\hat t\times\hat s} V_{Y,s} S_{Y,sr} W_{Y,r}^*\|_2
  = \|(I - Q_t Q_t^*) X|_{\hat t\times\hat s} V_{Y,s} S_{Y,sr} R_{Y,r}^*\|_2
\end{equation*}
and can replace $W_{Y,r}^*$ by $R_{Y,r}^*$ in (\ref{eq:original_product})
to obtain the new task of finding $Q_t$ with
\begin{align}\label{eq:condense_R}
  X|_{\hat t\times\hat s} V_{Y,s} S_{Y,sr} R_{Y,r}^*
  &\approx Q_t Q_t^* X|_{\hat t\times\hat s} V_{Y,s} S_{Y,sr} R_{Y,r}^* &
  &\text{ for all } s\in\mathcal{C}_t,\ (s,r)\in\lfaJK.
\end{align}
Each of these matrices has only at most $k$ columns, so we should be able
to handle the computation efficiently.

We can even go one step further:
just as we have taken advantage of the fact that $W_{Y,r}$ has only
$k$ columns, we can also use the fact that $V_{Y,s}$ also has only
$k$ columns.
To this end, we collect the column clusters of all admissible
blocks $(s,r)\in\lfaJK$ in a set
\begin{equation*}
  \brow(s) := \{ r\in\ctK\ :\ (s,r)\in\lfaJK \},
\end{equation*}
and enumerating $\brow(s) = \{ r_1, \ldots, r_\ell \}$ we can
introduce
\begin{equation}\label{eq:Gs_condensation}
  G_s := \begin{pmatrix}
    R_{Y,r_1} S_{Y,sr_1}^*\\
    \vdots\\
    R_{Y,r_\ell} S_{Y,sr_\ell}^*
  \end{pmatrix}
\end{equation}
and see that
\begin{align}\label{eq:condense_A}
  X|_{\hat t\times\hat s} V_{Y,s} G_s^*
  &\approx Q_t Q_t^* X|_{\hat t\times\hat s} V_{Y,s} G_s^* &
  &\text{ for all } s\in\mathcal{C}_t
\end{align}
is equivalent with (\ref{eq:condense_R}), since all submatrices of
(\ref{eq:condense_R}) also appear in (\ref{eq:condense_A}).

Since $V_{Y,s}$, and therefore also $G_s$, has only $k$
columns, we can use a thin Householder factorization to find an isometric
matrix $P_s$ and a small matrix $Z_s\in\bbbr^{\ell\times k}$ with
$G_s = P_s Z_s$, $\ell\leq k$.
Again we exploit the fact that we can apply orthogonal transformations
from the right without changing the approximation properties to obtain
\begin{align}\label{eq:condense_Z}
  X|_{\hat t\times\hat s} V_{Y,s} Z_s^*
  &\approx Q_t Q_t^* X|_{\hat t\times\hat s} V_{Y,s} Z_s^* &
  &\text{ for all } s\in\mathcal{C}_t.
\end{align}
This is still equivalent with (\ref{eq:condense_A}), but now
we only have one matrix with not more than $k$ columns for every
$s\in\mathcal{C}_t$.
If we used this formulation directly to construct $Q_t$, we would
violate the condition (\ref{eq:transfer}), since $Q_t$ would only
take care of submatrices connected directly to the cluster $t$, but
not to its ancestors.

Fortunately, this problem can be fixed easily by including the
ancestors' contributions by a recursive procedure:
with a top-down recursion, we can assume that $P_{s^+}$ and $Z_{s^+}$
for the parent $s^+$ and all ancestors of a cluster $s$ have already
been computed, and we can include the weight for the parent in the
new matrix
\begin{equation*}
  \widehat{G}_s := \begin{pmatrix}
    Z_{s^+} E_{Y,s}^*\\
    G_s
  \end{pmatrix}.
\end{equation*}
Multiplying the first block with $P_{s^+}$ will give us all admissible
blocks connected to ancestors of $s$, while the second block adds
the admissible blocks connected directly to $s$.
The resulting weight matrices $(Z_s)_{s\in\ctJ}$ are known as the
\emph{total weights} \cite[Chapter~6.6]{BO10} of the cluster
basis $(V_{Y,s})_{s\in\ctJ}$, since they measure how important the
different basis vectors are for the approximation of the entire matrix.

%
% Figure: Construction of total weights
%
\begin{figure}
  \begin{quotation}
    \begin{tabbing}
      \textbf{procedure} total\_weights($s$, $Z_{s^+}$);\\
      \textbf{begin}\\
      \quad\= $\widehat{G}_s \gets
               \begin{pmatrix}
                 Z_{s^+} E_{Y,s}^*\\
                 R_{Y,r_1} S_{Y,sr_1}^*\\
                 \vdots\\
                 R_{Y,r_\ell} Y_{Y,sr_\ell}^*
               \end{pmatrix}$ with
               $\{ r_1,\ldots,r_\ell \} = \brow(s)$;\\
      \> Compute a thin Householder
               factorization $P_s Z_s = \widehat{G}_s$;\\
      \> \textbf{for} $s'\in\chil(s)$ \textbf{do}
           total\_weights($s'$, $Z_s$)\\
      \textbf{end}
    \end{tabbing}
  \end{quotation}
  \caption{Construction of the total weights for the matrix $Y$}
  \label{fi:total_weights}
\end{figure}

%
% Definition: Total weights
%
\begin{definition}[Total weights]
\label{de:total_weights}
There are a family $(Z_s)_{s\in\ctJ}$ of matrices with $k$ columns and
no more than $k$ rows and a family $(P_s)_{s\in\ctJ}$ of isometric matrices
such that
\begin{align*}
  P_s Z_s &= \begin{cases}
    G_s &\text{ if } s \text{ is the root}\\
    \widehat{G}_s &\text{ otherwise}
  \end{cases} &
  &\text{ for all } s\in\ctJ.
\end{align*}
The matrices $(Z_s)_{s\in\ctJ}$ are called the \emph{total weights} for
the cluster basis $(V_{Y,s})_{s\in\ctJ}$ and the matrix $Y$.
\end{definition}

As mentioned above, the total weights can be computed by a top-down
recursion starting at the root and working towards the leaves of the
cluster tree $\ctJ$.
Under standard assumptions, this requires $\mathcal{O}(n k^2)$ operations
\cite[Algorithm~28]{BO10}, cf. Figure~\ref{fi:total_weights}.

Using basis weights and total weights, we have sufficiently reduced the
dimension of the matrices to develop the compression algorithm.

\paragraph*{Leaf clusters.}
We first consider the special case that $t\in\ctI$ is a leaf cluster.
In this case we may assume that $\hat t$ contains only a small number
of indices, so we can afford to set up the matrix
\begin{equation}\label{eq:induced_V}
  V_t := \begin{pmatrix}
    V_{X,t} &
    X|_{\hat t\times\hat s_1} V_{Y,s_1} Z_{s_1}^* & \ldots &
    X|_{\hat t\times\hat s_m} V_{Y,s_m} Z_{s_m}^*
  \end{pmatrix}
\end{equation}
directly, where we again enumerate $\mathcal{C}_t = \{ s_1,\ldots,s_m \}$
for ease of presentation.

We are looking for a low-rank approximation of this matrix that
can be used to treat sub-products appearing in the multiplication
algorithm.
Since we want to guarantee a given accuracy, the first block $V_{X,t}$
poses a challenge:
in expressions like (\ref{eq:semiuniform2}), it will be multiplied
by potentially hierarchically structured matrices $S_{X,ts} B_{sr}^*$, and
deriving a suitable weight matrix could be very complicated.

We solve this problem by simply ensuring that the range of $V_{X,t}$ is
left untouched by our approximation.
We find a thin Householder factorization
\begin{equation*}
  V_{X,t} = Q_{X,t} \begin{pmatrix} R_{X,t}\\ 0 \end{pmatrix}
\end{equation*}
with a $k_1\times k$ matrix $R_{X,t}$, $k_1=\min\{k,\#\hat t\}$, and
an orthonormal matrix $Q_{X,t}$.
Multiplying the entire matrix $V_t$ with $Q_{X,t}^*$ yields
\begin{equation*}
  Q_{X,t}^* V_t
  = \begin{pmatrix}
      R_{X,t} & M_{1,1} & \ldots & M_{1,m}\\
      0 & M_{2,1} & \ldots & M_{2,m}
    \end{pmatrix}
\end{equation*}
with submatrices given by
\begin{align*}
  Q_{X,t}^* X|_{\hat t\times\hat s_i} V_{Y,s_i} Z_{s_i}^*
  &= \begin{pmatrix}
       M_{1,i}\\ M_{2,i}
     \end{pmatrix} &
  &\text{ for all } i\in\{1,\ldots,m\}.
\end{align*}
Now we compute the singular value decomposition of the remainder
\begin{equation}\label{eq:induced_Vtilde}
  \widetilde{V}_t :=
  \begin{pmatrix}
    M_{2,1} & \ldots & M_{2,m}
  \end{pmatrix}
\end{equation}
and choose $k_2$ left singular vectors of $\widetilde{V}_t$ as the columns
of an isometric matrix $\widetilde{Q}_t$ to ensure
\begin{equation*}
  \widetilde{Q}_t \widetilde{Q}_t^* \widetilde{V}_t
  \approx \widetilde{V}_t.
\end{equation*}
We let
\begin{equation*}
  Q_t := Q_{X,t} \begin{pmatrix} I & 0 \\ 0 & \widetilde{Q}_t \end{pmatrix}
\end{equation*}
and observe
\begin{equation*}
  Q_t Q_t^* V_t
  = Q_{X,t} \begin{pmatrix}
               I & 0\\ 0 & \widetilde{Q}_t \widetilde{Q}_t^*
            \end{pmatrix} Q_{X,t}^* V_t
  = Q_{X,t} \begin{pmatrix}
               R_{X,t} & M_{1,1} & \ldots & M_{1,m}\\
               0 & \widetilde{Q}_t \widetilde{Q}_t^* M_{2,1} & \ldots &
               \widetilde{Q}_t \widetilde{Q}_t^* M_{2,m}
           \end{pmatrix},
\end{equation*}
so the matrix $V_{X,t}$ in the first $k$ columns of $V_t$ is indeed left
untouched, while the remainder of the matrix is approximated according to
the chosen singular vectors.
The rank of the new basis matrix $Q_t$ is now given by $k_t := k_1+k_2$.

In order to facilitate the next step of the procedure, we prepare the
auxiliary matrices
\begin{align*}
  A_{t,s} &:= Q_t^* X|_{\hat t\times\hat s} V_{Y,s} &
  &\text{ for all } s\in\mathcal{C}_t
\end{align*}
and keep the matrix $R_{X,t}$ describing the change of basis from
$V_{X,t}$ to $Q_t$.

\paragraph*{Non-leaf clusters.}
Now we consider the case that $t\in\ctI$ is not a leaf.
For ease of presentation, we assume that $t$ has exactly two children
$\chil(t)=\{t_1',t_2'\}$.
We handle this case by recursion and assume that the matrices $Q_{t'}$,
$R_{X,t'}$, and $A_{t',s'}$ for all children $t'\in\chil(t)$ and
$s'\in\mathcal{C}_{t'}$ have already been computed.

Since the new basis has to be nested according to (\ref{eq:transfer}),
we are not free to choose any matrix $Q_t$ in this case, but we have to
ensure that
\begin{equation}\label{eq:compression_nonleaf}
  Q_t = \begin{pmatrix}
          Q_{t_1'} E_{t_1'}\\
          Q_{t_2'} E_{t_2'}
        \end{pmatrix}
  = \underbrace{\begin{pmatrix}
                   Q_{t_1'} & \\
                   & Q_{t_2'}
                \end{pmatrix}}_{=: U_t}
    \underbrace{\begin{pmatrix}
                   E_{t_1'}\\ E_{t_2'}
                \end{pmatrix}}_{=: \widehat{Q}_t} = U_t \widehat{Q}_t
\end{equation}
holds with suitable transfer matrices $E_{t_1'}$ and $E_{t_2'}$.
We want $Q_t$ to be isometric, and since $Q_{t_1'}$ and $Q_{t_2'}$ can
be assumed to be isometric already, we have to ensure
\begin{equation*}
  I = Q_t^* Q_t = \widehat{Q}_t^* U_t^* U_t \widehat{Q}_t
    = \widehat{Q}_t^*
      \begin{pmatrix}
        Q_{t_1'}^* Q_{t_1'} & \\
        & Q_{t_2'}^* Q_{t_2'}
      \end{pmatrix}
      \widehat{Q}_t = \widehat{Q}_t^* \widehat{Q}_t.
\end{equation*}
Following \cite[Chapter~6.4]{BO10}, we can only approximate what can be
represented in the children's bases, i.e., in the range of $U_t$, so we
replace $V_t$ by its orthogonal projection
\begin{equation}\label{eq:induced_Vhat}
  \widehat{V}_t := U_t^* V_t
  = \begin{pmatrix}
      U_t^* V_{X,t} &
      U_t^* X|_{\hat t\times\hat s_1} V_{Y,s_1} Z_{s_1}^* &
      \ldots &
      U_t^* X|_{\hat t\times\hat s_m} V_{Y,s_m} Z_{s_m}^*
    \end{pmatrix}.
\end{equation}
Setting up the first matrix is straightforward, since we have
$R_{X,t_1'}$ and $R_{X,t_2'}$ at our disposal and obtain
\begin{equation*}
  \widehat{V}_{X,t} := U_t^* V_{X,t}
  = \begin{pmatrix}
      Q_{t_1'}^* & \\
      & Q_{t_2'}^*
    \end{pmatrix}
    \begin{pmatrix}
      V_{X,t_1'} E_{X,t_1'}\\
      V_{X,t_2'} E_{X,t_2'}
    \end{pmatrix}
  = \begin{pmatrix}
      R_{X,t_1'} E_{X,t_1'}\\
      0\\
      R_{X,t_2'} E_{X,t_2'}\\
      0
    \end{pmatrix},
\end{equation*}
so this block can be computed in $\mathcal{O}(k^3)$ operations.

Dealing with the remaining blocks is a little more challenging.
Let $s\in\mathcal{C}_t$.
By definition $(t,s)\in\ctIJ$ is not admissible, so $X|_{\hat t\times\hat s}$
has to be divided into submatrices.
For ease of presentation, we assume that we have $\chil(s)=\{s_1',s_2'\}$ and
therefore
\begin{equation*}
  X|_{\hat t\times\hat s}
  = \begin{pmatrix}
       X|_{\hat t_1'\times\hat s_1'} &
       X|_{\hat t_1'\times\hat s_2'}\\
       X|_{\hat t_2'\times\hat s_1'} &
       X|_{\hat t_2'\times\hat s_2'}
    \end{pmatrix}.
\end{equation*}
If $(t',s')\in\chil(t)\times\chil(s)$ is admissible, we can use the
basis-change $R_{X,t'}$ to get
\begin{align*}
  X|_{\hat t'\times\hat s'}
  &= V_{X,t'} S_{X,t's'} W_{X,s'}^*,\\
  A_{t',s'} := Q_{t'}^* X|_{\hat t'\times\hat s'} V_{Y,\hat s'}
  &= \begin{pmatrix}
       R_{X,t'} S_{X,t's'} W_{X,s'}^* V_{Y,s'}\\
       0
     \end{pmatrix}.
\end{align*}
To compute these matrices efficiently, we require the
\emph{cluster basis products}
\begin{align*}
  P_{XY,s'} &:= W_{X,s'}^* V_{Y,s'} &
  &\text{ for all } s'\in\ctJ
\end{align*}
that can be computed in advance by a simple recursive procedure
\cite[Chapter~5.3]{BO10} in $\mathcal{O}(n k^2)$ operations so that
the computation of
\begin{equation*}
  A_{t',s'} := \begin{pmatrix}
                 R_{X,t'} S_{X,t's'} P_{XY,s'}\\
                 0
              \end{pmatrix}
\end{equation*}
requires only $\mathcal{O}(k^3)$ operations.

If, on the other hand, $(t',s')\in\chil(t)\times\chil(s)$ is not admissible,
we have $s'\in\mathcal{C}_{t'}$ by definition and therefore can rely on
the matrix $A_{t',s'} = Q_{t'}^* X|_{\hat t'\times\hat s'} V_{Y,s'}$ to have been
prepared during the recursion for $t'\in\chil(t)$.

Once all submatrices are at our disposal, we can use
\begin{align}
  \widehat{A}_{t,s} &:= U_t^* X|_{\hat t\times\hat s} V_{Y,s}
   = \begin{pmatrix}
       Q_{t_1'}^* & \\
       & Q_{t_2'}^*
     \end{pmatrix}
     \begin{pmatrix}
       X|_{\hat t_1'\times\hat s_1'} &
       X|_{\hat t_1'\times\hat s_2'} \\
       X|_{\hat t_2'\times\hat s_1'} &
       X|_{\hat t_2'\times\hat s_2'}
     \end{pmatrix}
     \begin{pmatrix}
       V_{Y,s_1'} E_{Y,s_1'}\\
       V_{Y,s_2'} E_{Y,s_2'}
     \end{pmatrix}\notag\\
  &= \begin{pmatrix}
       Q_{t_1'}^* X|_{\hat t_1'\times\hat s_1'} V_{Y,s_1'} &
       Q_{t_1'}^* X|_{\hat t_1'\times\hat s_2'} V_{Y,s_2'}\\
       Q_{t_2'}^* X|_{\hat t_2'\times\hat s_1'} V_{Y,s_1'} &
       Q_{t_2'}^* X|_{\hat t_2'\times\hat s_2'} V_{Y,s_2'}
     \end{pmatrix}
     \begin{pmatrix}
       E_{Y,s_1'}\\ E_{Y,s_1'}
     \end{pmatrix}
   = \begin{pmatrix}
       A_{t_1',s_1'} & A_{t_1',s_2'}\\
       A_{t_2',s_1'} & A_{t_2',s_2'}
     \end{pmatrix}
     \begin{pmatrix}
       E_{Y,s_1'}\\
       E_{Y,s_2'}
     \end{pmatrix}\label{eq:induced_Ahat}
\end{align}
to set up the entire matrix
\begin{equation*}
  \widehat{V}_t = \begin{pmatrix}
    \widehat{V}_{X,t} &
    \widehat{A}_{t,s_1} Z_{s_1}^* & \ldots &
    \widehat{A}_{t,s_m} Z_{s_m}^*
  \end{pmatrix}
\end{equation*}
in $\mathcal{O}(k^3)$ operations as long as $m$ is bounded.

As in the case of leaf matrices, we want to preserve the row cluster
basis $V_{X,t}$ exactly, since we do not have reliable weights available
for this part of the matrix.
We can approach this task as before: we construct a Householder
factorization
\begin{equation*}
  \widehat{V}_{X,t} = \widehat{Q}_{X,t}
    \begin{pmatrix} R_{X,t}\\ 0 \end{pmatrix},
\end{equation*}
of $\widehat{V}_{X,t}$,
transform $\widehat{V}_t$ to $\widehat{Q}_{X,t}^* \widehat{V}_t$, and
apply approximation only to the lower half of the latter matrix.
This yields an isometric matrix $\widehat{Q}_t$ with
$\widehat{Q}_t \widehat{Q}_t^* \widehat{V}_{X,t} = \widehat{V}_{X,t}$
and $\widehat{Q}_t \widehat{Q}_t^* \widehat{V}_t \approx \widehat{V}_t$.
According to (\ref{eq:compression_nonleaf}), we can extract the
transfer matrices from $\widehat{Q}_t$ and define $Q_t$.

In order to satisfy the requirements of the recursion, we also
have to prepare the matrices $A_{t,s}$ for $s\in\mathcal{C}_t$.
Due to
\begin{align*}
  A_{t,s} &= Q_t^* X|_{\hat t\times\hat s} V_{Y,s}
  = \widehat{Q}_t^* U_t^* X|_{\hat t\times\hat s} V_{Y,s}
  = \widehat{Q}_t^* \begin{pmatrix}
      A_{t_1',s_1'} & A_{t_1',s_2'}\\
      A_{t_2',s_1'} & A_{t_2',s_2'}
    \end{pmatrix}
    \begin{pmatrix}
      E_{Y,s_1'}\\ E_{Y,s_2'}
    \end{pmatrix} = \widehat{Q}_t^* \widehat{A}_{t,s}
\end{align*}
this task can also be accomplished in $\mathcal{O}(k^3)$ operations.

%
% Figure: Induced basis compression
%
\begin{figure}
  \begin{quotation}
    \begin{tabbing}
      \textbf{procedure} induced\_basis($t$);\\
      \textbf{begin}\\
      \quad\= \textbf{if} $\chil(t)=\emptyset$ \textbf{then begin}\\
      \> \quad\= Set up $V_t$ as in (\ref{eq:induced_V});\\
      \> \> Construct a thin Householder
           factorization $Q_{X,t} R_{X,t} = V_{X,t}$;\\
      \> \> Compute $Q_{X,t}^* V_t$ and set up $\widetilde{V}_t$ as in
           (\ref{eq:induced_Vtilde});\\
      \> \> Compute the singular value decomposition of $\widetilde{V}_t$
           and choose a rank $k\in\bbbn$;\\
      \> \> Build $Q_t$ from the first $k$ left singular vectors;\\
      \> \> \textbf{for} $s\in\mathcal{C}_t$ \textbf{do}
           $A_{t,s} \gets Q_t^* X|_{\hat t\times\hat s} V_{Y,s}$\\
      \> \textbf{end}\\
      \> \textbf{else begin}\\
      \> \> \textbf{for} $t'\in\chil(t)$ \textbf{do}
               induced\_basis($t'$);\\
      \> \> Set up $\widehat{A}_{t,s}$ and $\widehat{V}_t$ as
               in (\ref{eq:induced_Ahat}) and (\ref{eq:induced_Vhat});\\
      \> \> Construct a thin Householder
              factorization $\widehat{Q}_{X,t} R_{X,t} = \widehat{V}_{X,t}$;\\
      \> \> Compute $\widehat{Q}_{X,t}^* \widehat{V}_t$ and set up
              $\widetilde{V}_t$;\\
      \> \> Compute the singular value decomposition of $\widetilde{V}_t$
              and choose a rank $k\in\bbbn$;\\
      \> \> Build $\widehat{Q}_t$ from the first $k$ left singular vectors
              and set transfer matrices;\\
      \> \> \textbf{for} $s\in\mathcal{C}_t$ \textbf{do}
               $A_{t,s} \gets \widehat{Q}_t^* \widehat{A}_{t,s}$\\
      \> \textbf{end}\\
      \textbf{end}
    \end{tabbing}
  \end{quotation}

  \caption{Construction of a compressed induced row basis for an
     $\mathcal{H}^2$-matrix product}
  \label{fi:induced_basis}
\end{figure}

% ============================================================
% Error control
% ============================================================
\section{Error control}
\label{se:error_control}

Let us consider a block $B_{tsr} = X|_{\hat t\times\hat s} Y|_{\hat s\times\hat r}$
of the product where $(s,r)\in\lfaJK$ is an admissible leaf.
We are interested in controlling the error
\begin{equation*}
  \|B_{tsr} - Q_t Q_t^* B_{tsr}\|_2
\end{equation*}
introduced by switching to the compressed induced row basis $Q_t$.

Since the errors introduced by the compression algorithm are
pairwise orthogonal \cite[Theorem~6.16]{BO10}, we can easily
obtain error estimates for individual blocks by controlling the
errors incurred for this block on all levels of the algorithm
\cite[Chapter~6.8]{BO10}.

We are interested in finding a block-relative error bound.
Since $(s,r)\in\lfaJK$ is assumed to be admissible, we have
\begin{gather*}
  B_{tsr} = X|_{\hat t\times\hat s} Y|_{\hat s\times\hat r}
  = X|_{\hat t\times\hat s} V_{Y,s} S_{Y,sr} W_{Y,r}^*,\\
  \|B_{tsr}\|_2 \leq \|X|_{\hat t\times\hat s} V_{Y,s}\|_2
                     \|S_{Y,sr} W_{Y,r}^*\|_2.
\end{gather*}
Both quantities on the right-hand side are easily available to
us: the matrices $A_{t,s} = U_t^* X|_{\hat t\times\hat s} V_{Y,s}$
appear naturally in our algorithm and we can use
\begin{equation*}
  \|X|_{\hat t\times\hat s} V_{Y,s}\|_2
  \geq \|U_t U_t^* X|_{\hat t\times\hat s} V_{Y,s}\|_2
  = \|U_t^* X|_{\hat t\times\hat s} V_{Y,s}\|_2
\end{equation*}
as a lower bound for the first term since $U_t$ is isometric.
For the second term we have
\begin{equation*}
  \|S_{Y,sr} W_{Y,r}^*\|_2
  = \|S_{Y,sr} R_{Y,r}^* Q_{Y,r}^*\|_2
  = \|S_{Y,sr} R_{Y,r}^*\|_2
\end{equation*}
by using the basis weights $R_{Y,r}$ introduced in
Definition~\ref{de:basis_weights}, so we can compute the norms of
these small matrices exactly.

Following the strategy presented in \cite[Chapter~6.8]{BO10}, we
would simply scale the blocks of $V_t$ and $\widehat{V}_t$ by the
reciprocals of the norm to ensure block-relative error estimates.

Since the two factors of the norm appear at different stages of the
algorithm and blocks are mixed irretrievably during the setup of the total
weight matrices, we have to follow a slightly different approach:
during the construction of the total weight matrices, we scale
the blocks of $A_s$ by the reciprocal of $\|S_{Y,sr} R_{Y,r}^*\|_2$,
and during the compression algorithm, we scale the blocks of
$V_t$ and $\widehat{V}_t$ by the reciprocals of
$\|X|_{\hat t\times\hat s} V_{Y,s}\|_2$ and
$\|\widehat{A}_{ts}\|_2 = \|U_t^* X|_{\hat t\times\hat s} V_{Y,s}\|_2
\leq \|X|_{\hat t\times\hat s} V_{Y,s}\|_2$, respectively.

Following the concept of \cite[Chapter~6.8]{BO10}, this leads to a
practical error-control strategy, but only yields bounds of the type
\begin{equation*}
  \|X|_{\hat t\times\hat s} Y|_{\hat s\times\hat r}
    - Q_t Q_t^* X|_{\hat t\times\hat s} Y|_{\hat s\times\hat r}\|_2
  \leq \epsilon \|X|_{\hat t\times\hat s}\|_2 \|Y|_{\hat s\times\hat r}\|_2
\end{equation*}
for a given tolerance $\epsilon\in\bbbr_{>0}$ and not the preferable
block-relative bound
\begin{equation*}
  \|X|_{\hat t\times\hat s} Y|_{\hat s\times\hat r}
    - Q_t Q_t^* X|_{\hat t\times\hat s} Y|_{\hat s\times\hat r}\|_2
  \leq \epsilon \|X|_{\hat t\times\hat s} Y|_{\hat s\times\hat r}\|_2.
\end{equation*}

% ============================================================
% Coarsening
% ============================================================
\section{Coarsening}
\label{se:coarsening}

\begin{figure}
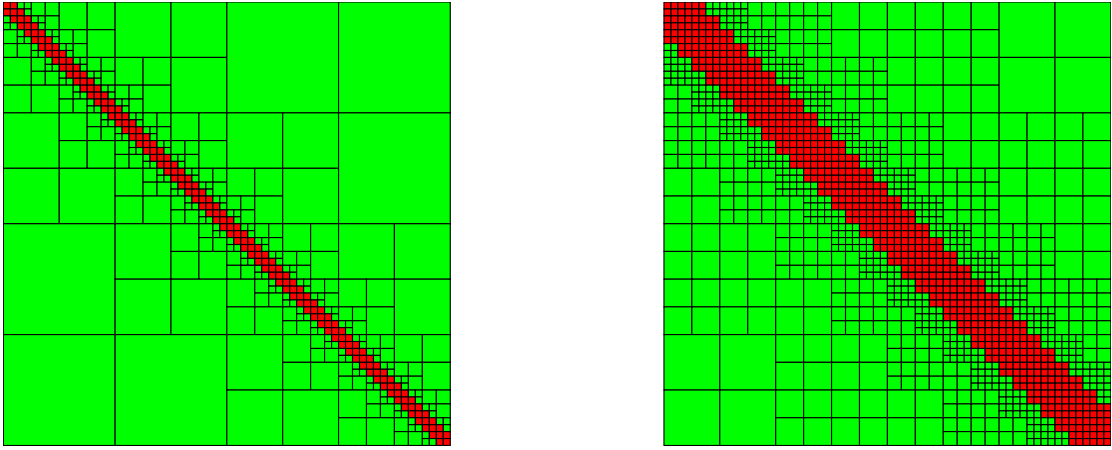

  \pgfdeclareimage[width=0.4\textwidth]{simpleblock}{fi_simpleblock}
  \pgfdeclareimage[width=0.4\textwidth]{productblock}{fi_productblock}

  \begin{centering}
    \pgfuseimage{simpleblock}%
    \hfill%
    \pgfuseimage{productblock}
  \end{centering}

  \caption{Original block structure (left) and
     product block structure (right)}
  \label{fi:blocks}
\end{figure}

The representations (\ref{eq:semiuniform}) and (\ref{eq:semiuniform2})
crucial to our algorithm hold only if either $(s,r)$ or $(t,s)$ are
admissible, i.e., the block partition used for the product $X Y$ has
to be sufficiently refined.
Assuming that $\ctIJ$ and $\ctJK$ are the block trees for the factors
$X$ and $Y$, the correct block tree $\ctprod$ for the product is given
inductively as the minimal block tree with
$\treeroot(\ctprod)=(\treeroot(\ctI),\treeroot(\ctK))$ and
\begin{align*}
  \chil(t,r) &= \begin{cases}
    \chil(t)\times\chil(r)
    &\text{ if } s\in\ctJ \text{ exists with }
        (t,s)\in\ctIJ\setminus\lfIJ,\\
    &\qquad\qquad\qquad \text{ and } (s,r)\in\ctJK\setminus\lfJK,\\
    \emptyset &\text{ otherwise},
  \end{cases}
\end{align*}
for all $(t,r)\in\ctprod$, i.e., we have to keep subdividing blocks as
long as there is at least one $s\in\ctJ$ such that both $(t,s)$ and $(s,r)$
are not leaves of $\ctIJ$ and $\ctJK$, respectively.
We denote the leaves of $\ctprod$ by $\lfprod$.
A leaf is considered admissible if for all $s\in\ctJ$
with $(t,s)\in\ctIJ$ and $(s,r)\in\ctJK$ at least one of $(t,s)$
and $(s,r)$ is admissible.
The admissible leaves are denoted by $\lfaprod$, the remaining
inadmissible leaves by $\lfiprod$.

The block tree $\ctprod$ induced by the multiplication is frequently
much finer than the prescribed block tree $\ctIK$ we
would like to use to approximate the product.
Figure~\ref{fi:blocks} shows a simple one-dimensional example:
on the left we have the block structure of $\ctIJ$ and $\ctJK$,
on the right the induced block structure $\ctprod$.

Fortunately, it is possible to prove \cite[Section~2.2]{GRHA02}
in standard situations that every admissible block in $\ctIK$ is split
into only a bounded number $C_{XY}$ of sub-blocks in $\ctprod$, and this
fact still allows us to construct efficient algorithms for the
matrix multiplication.
This effect is visible in Figure~\ref{fi:blocks}: every block on
the left-hand side is split into at most $16$ sub-blocks on the
right-hand side.

In this section, we assume that an $\mathcal{H}^2$-matrix
approximation of the product $G=XY$ has been computed by the
algorithm presented in the previous section, i.e., $G$ is
an $\mathcal{H}^2$-matrix with the block tree $\ctprod$, the
compressed induced row basis $V = (V_t)_{t\in\ctI}$ and
the compressed induced column basis $W = (W_r)_{r\in\ctK}$.
Our algorithm ensures that both bases are isometric, so no
basis weights have to be taken into account.

Our task is to construct a more efficient $\mathcal{H}^2$-matrix
with a coarser block tree $\ctIK$, i.e., we have to find
adaptive row and column cluster bases for a good approximation
of the matrix $G$.
As before, we focus only on the construction of the row basis,
since the same procedure can be applied to the adjoint matrix
$G^*$ to find a column basis.

\paragraph{Outline of the basis construction.}
Our construction is motivated by the fundamental algorithm
presented in \cite[Chapter~6.4]{BO10} with adjustments similar to
\cite[Chapter~6.6]{BO10} to take advantage of the fact that the
matrix $G$ is already presented as an $\mathcal{H}^2$-matrix,
although with a finer block tree.
The new cluster basis will again be denoted by $Q=(Q_t)_{t\in\ctI}$,
and we will again aim for an isometric basis.

For a given cluster $t\in\ctI$, we have to ensure that all admissible
blocks $(t,r)\in\lfaIK$ are approximated well, i.e.,
\begin{align*}
  G|_{\hat t\times\hat r} &\approx Q_t Q_t^* G|_{\hat t\times\hat r} &
  &\text{ for all } (t,r)\in\lfaIK.
\end{align*}
We collect the column clusters of these matrices in the sets
\begin{align*}
  \brow(t) &:= \{ r\in\ctK\ :\ (t,r)\in\lfaIK \} &
  &\text{ for all } t\in\ctI.
\end{align*}
Since we are looking for a \emph{nested} basis, (\ref{eq:transfer})
requires us to take the ancestors of $t$ into account, i.e., if
$t^*\in\ctI$ is an ancestor of $t$ and $(t^*,r)\in\lfaIK$, we
have to approximate $G|_{\hat t\times\hat r}$, too.
We express this fact by collecting all column clusters ultimately
contributing to $t$ in the sets
\begin{align*}
  \brow^*(t) &:= \begin{cases}
    \brow(t) &\text{ if } t=\treeroot(\ctI),\\
    \brow(t) \cup \brow^*(t^+) &\text{ if } t\in\chil(t^+),\ t^+\in\ctI
  \end{cases}
\end{align*}
for all $t\in\ctI$.
Our goal is to find an isometric nested cluster basis $Q=(Q_t)_{t\in\ctI}$
with
\begin{align*}
  G|_{\hat t\times\hat r}
  &\approx Q_t Q_t^* G|_{\hat t\times\hat r} &
  &\text{ for all } t\in\ctI,\ r\in\brow^*(t).
\end{align*}
If we combine all submatrices into large matrices
\begin{align*}
  G_t &:= G|_{\hat t\times\mathcal{R}_t}, &
  \mathcal{R}_t &:= \bigcup \{ \hat r\ :\ r\in\brow^*(t) \} &
  &\text{ for all } t\in\ctI,
\end{align*}
this property is equivalent with
\begin{align*}
  G_t &\approx Q_t Q_t^* G_t &
  &\text{ for all } t\in\ctI.
\end{align*}
The construction of $Q_t$ proceeds again by recursion:
if $t\in\lfI$ is a leaf, we compute the singular values and left
singular vectors of $G_t$ and combine the first $k$ left singular vectors
as the columns of $Q_t$ to obtain our adaptive isometric basis with
$G_t \approx Q_t Q_t^* G_t$.

If $t\in\ctI\setminus\lfI$ is not a leaf, we again assume --- for the
sake of simplicity --- that there are exactly two children
$\{t_1',t_2'\}=\chil(t)$.
These children are treated first so that we have $Q_{t_1'}$ and
$Q_{t_2'}$ at our disposal.
Using the isometric matrix
\begin{equation}\label{eq:Ut_definition}
  U_t := \begin{pmatrix}
           Q_{t_1'} & \\
           & Q_{t_2'}
         \end{pmatrix},
\end{equation}
we can approximate $G_t \approx U_t U_t^* G_t$ by the orthogonal projection
into the children's spaces and only have to treat the reduced matrix
\begin{equation}\label{eq:Ghat_definition}
  \widehat{G}_t := U_t^* G_t
\end{equation}
containing the corresponding coefficients.
Now we can again compute the singular values and use the first $k$ left
singular vectors of $\widehat{G}_t$ to form an isometric matrix $\widehat{Q}_t$.
Since $U_t$ is also isometric, the same holds for $Q_t := U_t \widehat{Q}_t$,
and splitting
\begin{equation*}
  \widehat{Q}_t = \begin{pmatrix} E_{t_1'}\\ E_{t_2'} \end{pmatrix}
\end{equation*}
yields the transfer matrices for the new adaptive basis with
$\widehat{G}_t \approx \widehat{Q}_t \widehat{Q}_t^* \widehat{G}_t$
and ultimately $G_t \approx Q_t Q_t^* G_t$.
An in-depth error analysis can be found in \cite[Chapter~6]{BO10}.

\paragraph{Condensation.}
The key to designing an efficient algorithm is the \emph{condensation}
of the matrix $G_t$:
we are looking for a matrix $C_t\in\bbbr^{\hat t\times\ell}$ with a small
number $\ell\in\bbbn$ of columns and a thin isometric matrix
$P_t\in\bbbr^{\mathcal{R}_t\times\ell}$ such that $G_t = C_t P_t^*$.
Since $P_t$ is isometric, i.e., $P_t^* P_t = I$, we have
\begin{equation*}
  \|(I - Q_t Q_t^*) G_t\|_2
  = \|(I - Q_t Q_t^*) C_t P_t^*\|_2
  = \|(I - Q_t Q_t^*) C_t\|_2
\end{equation*}
for any matrix $Q_t$, therefore we can replace $G_t$ throughout our
algorithm with the condensed matrix $C_t$ without changing the result.
If $\ell$ is smaller than $\#\mathcal{R}_t$, working with $C_t$ instead
of $G_t$ is more efficient.

The matrix $C_t$ can be constructed by considering the different blocks
of $G_t$ individually.
The most important example occurs for $r\in\ctK$ with $(t,r)\in\lfaprod$:
by definition, we have
\begin{equation*}
  G_t|_{\hat t\times\hat r}
  = G|_{\hat t\times\hat r}
  = V_t S_{tr} W_r^*,
\end{equation*}
and since the column basis $W$ is isometric, we can simply replace
$G|_{\hat t\times\hat r}$ by the matrix $V_t S_{tr}$ with only $k$ columns.

If we have multiple $r_1,\ldots,r_\ell$ with $(t,r_i)\in\lfaprod$ for
all $i\in\{1,\ldots,\ell\}$, we find
\begin{equation*}
  G_t|_{\hat t\times(\hat r_1\cup\ldots\cup\hat r_\ell)}
  = \begin{pmatrix}
      G|_{\hat t\times\hat r_1} & \ldots & G|_{\hat t\times\hat r_\ell}
    \end{pmatrix}
  = V_t \begin{pmatrix}
      S_{tr_1} & \ldots & S_{tr_\ell}
    \end{pmatrix}
    \begin{pmatrix}
      W_{r_1}^* & & \\
      & \ddots & \\
      & & W_{r_\ell}^*
    \end{pmatrix},
\end{equation*}
and we can not only drop the isometric right-most block
matrix, we can also use a thin Householder factorization
\begin{align}\label{eq:Zt_condensation}
  P_t Z_t &= \begin{pmatrix}
                S_{tr_1}^*\\ \vdots\\ S_{tr_\ell}^*
             \end{pmatrix}, &
  Z_t^* P_t^*= \begin{pmatrix}
                S_{tr_1} & \ldots & S_{tr_\ell}
             \end{pmatrix},
\end{align}
with an isometric matrix $P_t$ and $Z_t\in\bbbr^{k\times k}$ to obtain
\begin{equation*}
  G_t|_{\hat t\times(\hat r_1\cup\ldots\cup\hat r_\ell)}
  = V_t Z_t^* P_t^*
    \begin{pmatrix}
      W_{r_1}^* & & \\
      & \ddots & \\
      & & W_{r_\ell}^*
    \end{pmatrix},
\end{equation*}
which allows us to replace the entire block by the small matrix $V_t Z_t^*$
with only $k$ columns.
This approach can be extended to all ancestors of a cluster $t$, since
(\ref{eq:transfer}) allows us to translate a weight matrix $Z_{t^+}$ for
the parent of $t$ into the weight matrix $Z_{t^+} E_{t}^*$ for $t$ as
in Definition~\ref{de:total_weights}.
Using this approach, \emph{all} admissible blocks in $\ctprod$ can be
reduced to small weight matrices $Z_t$ for $t\in\ctI$ and handled
very efficiently.

Unfortunately, $\ctprod$ is finer than the block tree $\ctIK$ we want
to use for the final result, so we have to be able to deal with
blocks $(t,r)\in\lfaIK$ that are not admissible in $\ctprod$.

If $(t,r)\in\lfiprod$ holds, our $\mathcal{H}^2$-matrix representation
of the product $G$ contains the matrix $G|_{\hat t\times\hat r}$
explicitly and we can simply copy it into $C_t$ during our algorithm.

The situation becomes more challenging if $(t,r)\in\ctprod\setminus\lfprod$
holds, i.e., if $(t,r)$ corresponds to a block that is subdivided in
$\ctprod$, but an admissible leaf in $\ctIK$.

Let us consider an example:
we assume that $\chil(t)=\{t_1',t_2'\}$ and $\chil(r)=\{r_1',r_2'\}$ and
that $(t_i',r_j')\in\lfaprod$ are admissible blocks for all $i,j\in\{1,2\}$.
Then the corresponding submatrix has the form
\begin{align*}
  G|_{\hat t\times\hat r}
  &= \begin{pmatrix}
       G|_{\hat t_1\times\hat r_1} &
       G|_{\hat t_1\times\hat r_2}\\
       G|_{\hat t_2\times\hat r_1} &
       G|_{\hat t_2\times\hat r_2}
     \end{pmatrix}
   = \begin{pmatrix}
       V_{t_1} S_{t_1 r_1} W_{r_1}^* &
       V_{t_1} S_{t_1 r_2} W_{r_2}^*\\
       V_{t_2} S_{t_2 r_1} W_{r_1}^* &
       V_{t_2} S_{t_2 r_2} W_{r_2}^*
     \end{pmatrix}\\
  &= \begin{pmatrix}
       V_{t_1} S_{t_1 r_1} & V_{t_1} S_{t_1 r_2}\\
       V_{t_2} S_{t_2 r_1} & V_{t_2} S_{t_2 r_2}
     \end{pmatrix}
     \begin{pmatrix}
       W_{r_1}^* & \\
       & W_{r_2}^*
     \end{pmatrix},
\end{align*}
and we can again drop the right-most term because
it is isometric, thus replacing $G|_{\hat t\times\hat r}$ by the matrix
\begin{equation*}
  \begin{pmatrix}
    V_{t_1} S_{t_1 r_1} & V_{t_1} S_{t_1 r_2}\\
    V_{t_2} S_{t_2 r_1} & V_{t_2} S_{t_2 r_2}
  \end{pmatrix}
\end{equation*}
with only $2k$ columns.

In order to extend this approach to arbitrary block structures
within $G|_{\hat t\times\hat r}$, we introduce the \emph{column tree}.

%
% Figure: Column tree
%
\begin{figure}
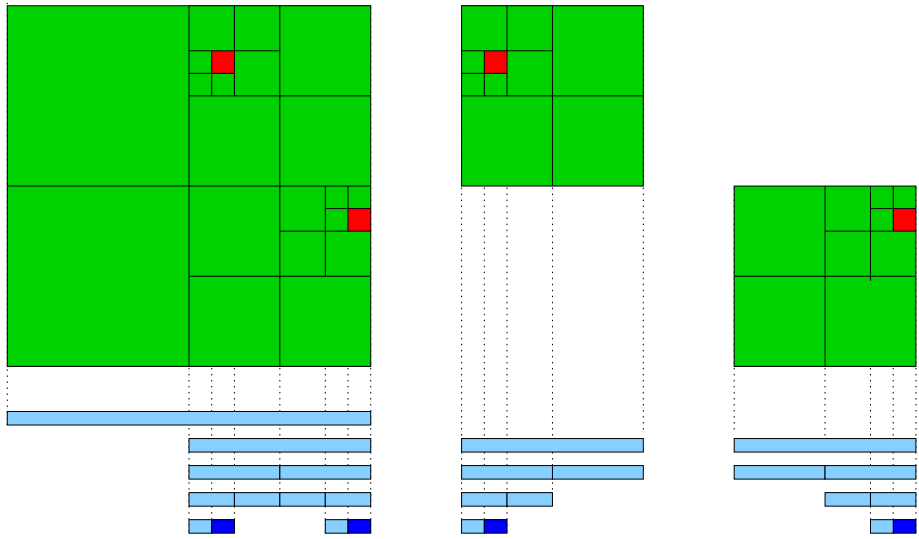

  \pgfdeclareimage[width=12cm]{column}{fi_column}

  \begin{center}
    \pgfuseimage{column}
  \end{center}
  \caption{Block structures and corresponding column trees}
  \label{fi:column}
\end{figure}

%
% Definition: Column tree
%
\begin{definition}[Column tree]
The \emph{column tree} for $(t,r)\in\ctprod$ is the minimal cluster tree
$\mathcal{T}_{(t,r)}$
\begin{itemize}
  \item with root $r$ and
  \item for every $r'\in\mathcal{T}_{(t,r)}$ there is a $t'\in\ctI$
    such that $(t',r')$ is a descendant of $(t,r)$ in the block tree $\ctprod$,
  \item every $r'\in\mathcal{T}_{(t,r)}$ is either a leaf or has
    the same children as it does in $\ctK$.
\end{itemize}
The leaves of the column tree $\mathcal{T}_{(t,r)}$ are denoted by
$\mathcal{L}_{(t,r)}$.
A leaf $r'\in\mathcal{L}_{(t,r)}$ is called \emph{inadmissible} if
a $t'\in\ctI$ with $(t',r')\in\lfiprod$ exists and \emph{admissible}
otherwise.
The corresponding index sets $\{ \hat r'\ :\ r'\in\mathcal{L}_{(t,r)} \}$
are a disjoint partition of the index set $\hat r$.

By construction $\mathcal{T}_{(t,r)}$ is a subtree of the cluster tree $\ctK$.
\end{definition}

The column tree $\mathcal{T}_{(t,r)}$ can be interpreted as an
``orthogonal projection'' of the subtree of $\ctprod$ rooted at $(t,r)$
into its second component.
Examples of block structures and the corresponding column trees
can be seen in Figure~\ref{fi:column}.

%
% Lemma: Column tree representation
%
\begin{lemma}[Column representation]
\label{le:column_representation}
Let $(t,r)\in\ctprod$, and let $(t',r')\in\lfaprod$ be a descendant
of $(t,r)$.
Let $r^*\in\mathcal{L}_{(t,r)}$ be a descendant of $r'$.
There is a matrix $A\in\bbbr^{\hat t'\times k}$ with
$G|_{\hat t\times\hat r^*} = A W_{r^*}^*$.
\end{lemma}
\begin{proof}
This is a simple consequence of (\ref{eq:transfer}):
due to $(t',r')\in\lfaprod$, we have
\begin{equation*}
  G|_{\hat t'\times\hat r'} = V_{t'} S_{t'r'} W_{r'}^*,
\end{equation*}
and restricting to a child $r''\in\chil(r')$ leads to
\begin{equation*}
  G|_{\hat t'\times\hat r''} = V_{t'} S_{t'r'} F_{r''}^* W_{r''}^*
\end{equation*}
due to (\ref{eq:transfer}).
We can proceed by induction until we reach the descendant $r^*$.
\end{proof}

This lemma suggests how to perform condensation for subdivided
matrices:
for $(t,r)\in\ctprod$, we enumerate the leaves
$\mathcal{L}_{(t,r)}=\{r_1,\ldots,r_\ell\}$ of the column tree
$\mathcal{T}_{(t,r)}$ and see that Lemma~\ref{le:column_representation}
gives us matrices $A_{r_1},\ldots,A_{r_\ell}\in\bbbr^{\hat t\times k}$ such that
\begin{equation*}
  G|_{\hat t\times\hat s}
  = \begin{pmatrix}
    A_{r_1} W_{r_1}^* & \ldots & A_{r_\ell} W_{r_\ell}^*
  \end{pmatrix}
  = \begin{pmatrix}
    A_{r_1} & \ldots A_{r_\ell}
  \end{pmatrix}
  \begin{pmatrix}
    W_{r_1}^* & & \\
    & \ddots & \\
    & & W_{r_\ell}^*
  \end{pmatrix},
\end{equation*}
where we can again drop the isometric right-most factor to get
a matrix with only $\ell k$ columns.

In order to construct these condensed representations efficiently,
we would like to use a recursive approach:
first we construct representations for children of a block, then
we merge them to get a representation for the parent.
Figure~\ref{fi:column} indicates a problem: different child blocks
within the same column may have different column trees, and a
common ``super tree'' has to be constructed to cover all of them.

Fortunately, we have again the nested structure (\ref{eq:transfer})
of the cluster basis $W$ at our disposal:
if we have $r$ as a leaf in one column tree, while another column tree
needs the children $\{r_1',r_2'\}=\chil(r)$ of $r$, we simply use
the transfer matrices $F_{r_1'}$ and $F_{r_2'}$ as in
\begin{equation*}
  A_r W_r^*
  = A_r \begin{pmatrix}
          W_{r_1'} F_{r_1'}\\
          W_{r_2'} F_{r_2'}
        \end{pmatrix}^*
  = \begin{pmatrix}
      A_r F_{r_1'}^* & A_r F_{r_2'}^*
    \end{pmatrix}
    \begin{pmatrix}
      W_{r_1'}^* & \\
      & W_{r_2'}^*
    \end{pmatrix}
\end{equation*}
to add representations for both children.
We can apply this procedure to ensure that the representation for
a column tree $\mathcal{T}_{(t',r')}$ of a sub-block $(t',r')$
matches a column tree $\mathcal{T}_{(t,r)}$ of its parent.
The corresponding algorithm is given in Figure~\ref{fi:match_column}.

%
% Figure: Match column trees
%
\begin{figure}
  \begin{quotation}
    \begin{tabbing}
      \textbf{procedure} match\_column($r$, $\mathcal{T}$, $A$);\\
      \textbf{begin}\\
      \quad\= \textbf{if} $\chil(r)\neq\emptyset$ \textbf{then begin}\\
      \> \quad\= \textbf{for} $r'\in\chil(r)$ \textbf{do}\\
      \> \> \quad\=  match\_column($r'$, $\mathcal{T}$, $A$);\\
      \> \textbf{end else if} $\chil_{\mathcal{T}}(r)\neq\emptyset$
               \textbf{then begin}\\
      \> \> $\chil(r) \gets \chil_{\mathcal{T}}(r)$;\\
      \> \> \textbf{for} $r'\in\chil(r)$ \textbf{do begin}\\
      \> \> \> $A_{r'} \gets A_r F_{r'}^*$;\\
      \> \> \> match\_column($r'$, $\mathcal{T}$, $A$)\\
      \> \> \textbf{end}\\
      \> \textbf{end else if} $\adm(r)$ and not $\adm_\mathcal{T}(r)$
      \textbf{do begin}\\
      \> \> $\adm(r) \gets \adm_\mathcal{T}(r)$;\\
      \> \> $A_r \gets A_r W_{r}^*$;\\
      \> \textbf{end}\\
      \textbf{end}
    \end{tabbing}
  \end{quotation}
  \caption{Extending a column tree with root $r$ to match a given
    column tree $\mathcal{T}$. $\chil(r)$ and $\adm(r)$ denote the
    children of $r$ and its admissibility in its own column tree,
    $\chil_{\mathcal{T}}(r)$ and $\adm_{\mathcal{T}}(r)$ denote the children
    and the admissibility in the target tree $\mathcal{T}$.}
  \label{fi:match_column}
\end{figure}

Now we can return to the construction of the new cluster basis $Q$.
As mentioned before, we proceed by recursion.
On the way from the root towards the leaves, we accumulate total
weight matrices $Z_t$ by computing thin Householder factorizations
\begin{equation*}
  P_t Z_t = \begin{pmatrix}
    Z_{t^+} E_t^*\\
    S_{t,r_1}^*\\
    \vdots\\
    S_{t,r_m}
  \end{pmatrix},
\end{equation*}
where $t^+$ denotes the parent of $t$ if it exists (this submatrix
vanishes in the root cluster) and $\{r_1,\ldots,r_m\}=\brow_{XY}(t)$ are
the column clusters of all admissible blocks with respect to the
block tree $\ctprod$, i.e.,
\begin{align*}
  \brow_{XY}(t) &:= \{ r\in\ctK\ :\ (t,r)\in\lfaprod \} &
  &\text{ for all } t\in\ctI.
\end{align*}
Using these weights, the entire part of $G_t$ or $\widehat{G}_t$ that
is admissible in $\ctprod$ can be condensed to $V_t Z_t^*$.
This leaves us with the part that is inadmissible in $\ctprod$, but
admissible in $\ctIK$.

\paragraph*{Leaf clusters.}
Let $t\in\lfI$ be a leaf of $\ctI$.
Since $t$ is a leaf, the only inadmissible blocks $(t,r)\in\lfiprod$
must correspond to nearfield matrices that are represented
explicitly by our construction, i.e., $G|_{\hat t\times\hat r}$ is
readily available to us.

Let $\{r_1,\ldots,r_m\}$ denote all of these clusters, i.e., we have
$(t,r_i)\in\lfiprod$ for all $i\in\{1,\ldots,m\}$ and
an ancestor of each block is admissible in $\ctIK$.
We use the condensed version $V_t Z_t^*$ of the admissible part in
combination with the inadmissible parts to get
\begin{equation}\label{eq:ct_leaf}
  C_t := \begin{pmatrix}
    V_t Z_t^* & G|_{\hat t\times\hat r_1} & \ldots
    & G|_{\hat t\times\hat r_m}
  \end{pmatrix}.
\end{equation}
As mentioned before, the number of inadmissible blocks in $\ctprod$
that are descendants of admissible blocks in $\ctIK$ is bounded,
therefore $C_t$ has $\#\hat t$ rows and $\mathcal{O}(k)$ columns.

This makes it sufficiently small to compute the singular value decomposition
of $C_t$ in $\mathcal{O}(k^2 \#\hat t)$ operations.
We can use the singular values to choose an appropriate rank controlling
the approximation error, and we can use the leading left singular vectors
to set up the isometric matrix $Q_t$.

In preparation of the next steps, we also compute the basis-change
matrix $R_t := Q_t^* V_t$ and the auxiliary matrices
\begin{align*}
  A_{t,r_i} &= Q_t^* G|_{\hat t\times r_i} &
  &\text{ for all } i\in\{1,\ldots,m\}.
\end{align*}

\paragraph*{Non-leaf clusters.}
Let now $t\in\ctI\setminus\lfI$ be a non-leaf cluster of $\ctI$.
For the sake of simplicity, we restrict out attention to the case
$\chil(t)=\{t_1',t_2'\}$.
Since we are using recursion, $Q_{t_1'}$ and $Q_{t_2'}$ are already
at our disposal, as are the basis changes $R_{t_1'}$ and $R_{t_2'}$ and
the weights $A_{t_1',r}$ and $A_{t_2',r}$ for the inadmissible blocks
connected to the children.

Our goal is to set up a condensed counterpart of the matrix
\begin{equation*}
  \widehat{G}_t := \begin{pmatrix}
    Q_{t_1'}^* G|_{\hat t_1'\times\mathcal{R}_t}\\
    Q_{t_2'}^* G|_{\hat t_2'\times\mathcal{R}_t}
  \end{pmatrix}.
\end{equation*}
Again, all the admissible parts are already represented by the
weight matrix $Z_t$, and we can replace them all by
\begin{equation*}
  \widehat{V}_t Z_t^* \qquad\text{with}\qquad
  \widehat{V}_t := \begin{pmatrix}
    Q_{t_1'}^* V_{t_1'} E_{t_1'}\\
    Q_{t_2'}^* V_{t_2'} E_{t_2'}
  \end{pmatrix}
  = \begin{pmatrix}
    R_{t_1'} E_{t_1'}\\
    R_{t_2'} E_{t_2'}
  \end{pmatrix}.
\end{equation*}
The inadmissible part is a little more challenging.
We consider a cluster $r\in\ctK$ with $(t,r)\in\ctprod\setminus\lfprod$,
i.e., $(t,r)$ is subdivided.
For the sake of simplicity, we assume $\chil(r)=\{r_1',r_2'\}$ and have
\begin{equation*}
  G|_{\hat t\times\hat r}
  = \begin{pmatrix}
    G|_{\hat t_1'\times\hat r_1'} & G|_{\hat t_1'\times\hat r_2'}\\
    G|_{\hat t_2'\times\hat r_1'} & G|_{\hat t_2'\times\hat r_2'}
  \end{pmatrix}.
\end{equation*}
We require only the projection into the ranges of $Q_{t_1'}$ and
$Q_{t_2'}$, i.e.,
\begin{equation*}
  \begin{pmatrix}
    Q_{t_1'}^* G|_{\hat t_1'\times\hat r_1'} &
    Q_{t_1'}^* G|_{\hat t_1'\times\hat r_2'}\\
    Q_{t_2'}^* G|_{\hat t_2'\times\hat r_1'} &
    Q_{t_2'}^* G|_{\hat t_2'\times\hat r_2'}
  \end{pmatrix}.
\end{equation*}
If one of these submatrices $Q_{t'}^* G|_{\hat t'\times\hat r'}$ is again
inadmissible, we have already seen it when treating the corresponding
child $t'$, and therefore $A_{t',r'}$ is already at our disposal.

%
% Figure: Compression of inadmissible blocks
%
\begin{figure}
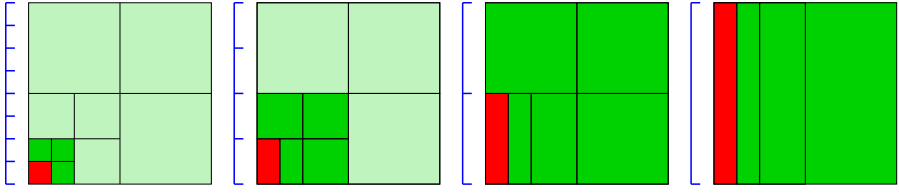

  \pgfdeclareimage[width=0.8\textwidth]{inadm}{fi_inadm}

  \begin{center}
    \pgfuseimage{inadm}
  \end{center}

  \caption{Coarsening of a subdivided block: on the finest level
    (left) the submatrices are compressed individually.
    Stepping to a coarser level (right), the submatrices are merged
    by matching the column trees.
    Light green blocks are included in the parent's weight $Z_t$.}
  \label{fi:iadm}
\end{figure}

If one of these submatrices $Q_{t'}^* G|_{\hat t'\times\hat r'}$ is, on the
other hand, admissible, we can use (\ref{eq:vsw}) and the basis-change
matrix $R_{t'}$ prepared previously to obtain
\begin{equation*}
  A_{t',r'} := Q_{t'}^* G|_{\hat t'\times\hat r'}
  = Q_{t'}^* V_{t'} S_{t'r'} W_{r'}^*
  = R_{t'} S_{t'r'} W_{r'}^*.
\end{equation*}
Since the right-most factor $W_{r'}$ is again
isometric, we can discard it if necessary without changing the
singular values and the left singular vectors.
Now that we have representations for all submatrices, we can use
the function ``match\_column'' of Figure~\ref{fi:column} to ensure
that they all share the same column tree $\mathcal{T}_{(t,r)}$ and
combine them into the condensed matrix
\begin{align*}
  \widehat{A}_{t,r} &:= \begin{pmatrix}
    A_{t_1',r_1} & \ldots & A_{t_1',r_\ell}\\
    A_{t_2',r_1} & \ldots & A_{t_2',r_\ell}
  \end{pmatrix}, &
  \mathcal{L}_{(t,r)} &= \{ r_1,\ldots,r_\ell \}.
\end{align*}
Combining the admissible and inadmissible submatrices yields
the condensed counterpart of $\widehat{G}_t$ in the form
\begin{equation}\label{eq:ct_nonleaf}
  \widehat{C}_t
  := \begin{pmatrix}
    \widehat{V}_t Z_t^* & \widehat{A}_{t,r_1} & \ldots
    & \widehat{A}_{t,r_m}
  \end{pmatrix},
\end{equation}
where $\{r_1,\ldots,r_m\}$ again denotes all column clusters
of blocks that are subdivided in $\ctprod$ and have an admissible
ancestor in $\ctIK$, i.e., $(t,r_i)\in\ctprod\setminus\lfprod$ for
all $i\in\{1,\ldots,m\}$.
As before, the number of inadmissible blocks in $\ctprod$ as descendants
of admissible blocks in $\ctIK$ is bounded, therefore the matrix
$\widehat{C}_t$ has only $\mathcal{O}(k)$ columns and not more
than $2k$ rows.

This makes it sufficiently small to compute the singular value
decomposition in $\mathcal{O}(k^3)$ operations, and we can again use
the singular values to choose an appropriate rank controlling the
approximation error and the leading left singular vectors to
set up an isometric matrix $\widehat{Q}_t$.

We can split $\widehat{Q}_t$ into transfer matrices for $t_1'$
and $t_2'$ to get
\begin{equation*}
  Q_t := U_t \widehat{Q}_t,
\end{equation*}
i.e., the new cluster basis matrix $Q_t$ is expressed via transfer matrices
as in (\ref{eq:transfer}).

We also compute the basis-change $R_t := Q_t^* V_t
= \widehat{Q}_t^* \widehat{V}_t$ and the auxiliary matrices
\begin{align*}
  A_{t,r_i} &= Q_t^* G|_{\hat t\times\hat r_i}
             = \widehat{Q}_t^* \widehat{A}_{t,r_i} &
  &\text{ for all } i\in\{1,\ldots,m\}
\end{align*}
in preparation for the next stage of the recursion.
We can see that all of these operations take no more than $\mathcal{O}(k^3)$
operations.
The entire algorithm is summarized in Figure~\ref{fi:basis}.

%
% Figure: Basis construction
%
\begin{figure}
  \begin{quotation}
    \begin{tabbing}
      \textbf{procedure} build\_basis($t$);\\
      \textbf{begin}\\
      \quad\= Compute $Z_t$ as in (\ref{eq:Zt_condensation});\\
      \> \textbf{if} $\chil(t)=\emptyset$ \textbf{then begin}\\
      \> \quad\= Set up $C_t$ as in (\ref{eq:ct_leaf});\\
      \> \> Compute its singular value decomposition
            and choose a rank $k\in\bbbn$;\\
      \> \> Build $Q_t$ from the first $k$ left singular vectors;\\
      \> \> $R_t \gets Q_t^* V_t$;\\
      \> \> \textbf{for} $r\in\ctK$ with $(t,r)\in\lfiprod$
            \textbf{do} $A_{t,r} \gets Q_t^* G|_{\hat t\times\hat r}$\\
      \> \textbf{end}\\
      \> \textbf{else begin}\\
      \> \> \textbf{for} $t'\in\chil(t)$ \textbf{do} build\_basis($t'$);\\
      \> \> Set up $\widehat{C}_t$ as in (\ref{eq:ct_nonleaf});\\
      \> \> Compute its singular value decomposition
            and choose a rank $k\in\bbbn$;\\
      \> \> Build $\widehat{Q}_t$ from the first $k$ left singular vectors
            and set transfer matrices;\\
      \> \> $R_t \gets \widehat{Q}_t^* \widehat{V}_t$;\\
      \> \> \textbf{for} $r\in\ctK$ with $(t,r)\in\ctprod\setminus\lfprod$
            \textbf{do} $A_{t,r} \gets \widehat{Q}_t^* \widehat{A}_{t,r}$\\
      \> \textbf{end}\\
      \textbf{end}
    \end{tabbing}
  \end{quotation}

  \caption{Recursive algorithm for constructing an adaptive row
    cluster basis using the coarse block tree $\ctIK$}
  \label{fi:basis}
\end{figure}

%
% Remark: Complexity
%
\begin{remark}[Complexity]
The basis construction algorithm has a complexity of $\mathcal{O}(n k^2)$
if the following conditions are met:
\begin{itemize}
  \item there are constants $c_\text{lf},C_\text{lf}$ with
    $c_\text{lf} k \leq \#\hat t \leq C_\text{lf} k$ for all leaf clusters
    $t\in\ctI$,
  \item the block tree $\ctIK$ is \emph{sparse}, i.e., there is
    a constant $C_\text{sp}$ with
    \begin{align*}
      \#\{ r\ :\ (t,r)\in\ctIK \} &\leq C_\text{sp} &
      &\text{ for all } t\in\ctI,
    \end{align*}
  \item the fine block tree $\ctprod$ is not too fine, i.e., there
    is a constant $C_{XY}$ with
    \begin{align*}
      \#\{ (t',r')\in\ctprod\ :\ 
           \hat t'\times\hat r'\subseteq\hat t\times\hat r \} &\leq C_{XY} &
      &\text{ for all } (t,r)\in\lfaIK.
    \end{align*}
\end{itemize}
The first condition ensures that the matrices $C_t$ and $\widehat{C}_t$
have only $\mathcal{O}(k)$ rows and that the cluster tree $\ctI$ has
only $\mathcal{O}(n/k)$ elements.
This can be ensured by a suitable construction of
the cluster trees.
The second and third conditions ensure that the matrices appearing in the
construction of the weight matrices (\ref{eq:Zt_condensation}) have
only $\mathcal{O}(k)$ columns.
These conditions also ensure that the matrices $C_t$ and $\widehat{C}_t$
have only $\mathcal{O}(k)$ columns.
If one of the standard admissibility conditions is
used, e.g., for a quasi-uniform grid, the second and third condition
are guaranteed.
Together, these three conditions ensure that only
$\mathcal{O}(k^3)$ operations are required for every cluster, and since
there are only $\mathcal{O}(n/k)$ clusters, we obtain $\mathcal{O}(n k^2)$
operations in total.
\end{remark}

%
% Table: Symbols used in the algorithms
%
\begin{table}
  \begin{tabular}{r|l}
    Symbol & Meaning\\
    \hline
    $V_{X,t}$, $V_{Y,s}$ & Row cluster bases for $X$ and $Y$,
              transfer matrices $E_{X,t}$, $E_{Y,s}$\\
    $W_{X,s}$, $W_{Y,r}$ & Column cluster basis for $X$ and $Y$,
              transfer matrices $F_{X,s}$, $F_{Y,r}$\\
    $R_{Y,r}$ & Basis weight matrices for $W_{Y,r}$\\
    $Z_s$ & Total weights for the matrix $Y$\\
    $Z_t$ & Total weights for the product $XY$\\
    $S_{X,ts}$, $S_{Y,sr}$ & Coupling matrices for $X$ and $Y$\\
    $V_t$, $W_r$ & Induced row and column cluster basis for the product $XY$\\
    $\mathcal{T}_{\Idx\times\Jdx}$, $\mathcal{T}_{\Jdx\times\Kdx}$,
    $\mathcal{T}_{\Idx\times\Kdx}$ & Block trees for $X$, $Y$, and $Z$\\
    $\mathcal{T}_{\Idx\times\Kdx}^{XY}$ & Induced block tree for $XY$
  \end{tabular}

  \caption{List of symbols}
\end{table}

% ============================================================
% Numerical experiments
% ============================================================
\section{Numerical experiments}

We investigate the practical performance of the new algorithm
by considering matrices appearing in the context of boundary
element methods:
the single-layer operator
\begin{align*}
  \mathcal{V}[u](x) &= \int_{\Gamma_S} \frac{1}{4\pi\|x-y\|}
                            u(y) \,dy &
  &\text{ for all } x\in\Gamma_S
\end{align*}
on the unit sphere $\Gamma_S := \{ x\in\bbbr^3\ :\ \|x\|_2=1 \}$,
approximated on surface meshes constructed by starting with a double
pyramid, splitting the faces regularly into triangles, and moving
their vertices to the sphere, 
and the double-layer operator
\begin{align*}
  \mathcal{K}[u](x) &= \int_{\Gamma_C} \frac{\langle n(x), x-y \rangle}
           {4\pi \|x-y\|^3} u(y) \,dy &
  &\text{ for all } x\in\Gamma_C
\end{align*}
on the surface of the cube $\Gamma_C := \partial[-1,1]^3$,
represented by surface meshes constructed by regularly splitting
its six faces into triangles.

We discretize both operators by Galerkin's method using piecewise
constant basis functions on the triangular mesh.
The resulting matrices are approximated using hybrid cross
approximation \cite{BOGR04} and converting the resulting
hierarchical matrices into $\mathcal{H}^2$-matrices
\cite[Chapter~6.5]{BO10}.
For the single-layer operator on the unit sphere, we obtain
matrices of dimensions between $2\,048$ and $2\,097\,152$,
while we have dimensions between $3\,072$ and $3\,145\,728$
for the double-layer operator on the cube's surface.

For the compression of the induced row and column basis, we
use a block-relative accuracy of $10^{-4}$ as described in
Section~\ref{se:error_control}.
For the re-compression into an $\mathcal{H}^2$-matrix with
the final coarser block tree $\ctIK$, we use the same
accuracy in combination with the block-relative error control
strategy described in \cite[Chapter~6.8]{BO10}.
For the largest matrix, the row and column cluster
bases for the original matrix require $719.7$ MB each, the compressed
induced bases require $2\,997.7$ MB each, and the final bases after
coarsening $972.8$ and $973.4$ MB.
The relatively high storage requirements for the intermediate bases
coincide with an accuracy that is far higher than prescribed.

%
% Table: Multiplication of the single-layer matrix
%
\begin{table}
  \begin{equation*}
    \begin{array}{r|rrrr|rrrr}
      & \multicolumn{4}{|c}{\text{Induced}} & \multicolumn{4}{|c}{\text{Final}}\\
    n & t_\text{row} & t_\text{col} & t_\text{mat} & \epsilon_2
      & t_\text{row} & t_\text{col} & t_\text{mat} & \epsilon_2\\
      \hline
      2\,048 & 0.1 & 0.1 & 0.3 & 1.12_{-7} & 0.4 & 0.2 & 0.0 & 1.12_{-5}\\
      4\,608 & 0.2 & 0.2 & 0.6 & 9.41_{-8} & 1.3 & 0.7 & 0.1 & 5.91_{-6}\\
      8\,192 & 0.4 & 0.4 & 1.3 & 1.20_{-7} & 2.6 & 1.5 & 0.2 & 4.87_{-6}\\
     18\,432 & 1.0 & 1.0 & 3.1 & 1.13_{-7} & 6.6 & 3.8 & 0.5 & 4.93_{-6}\\
     32\,768 & 1.8 & 1.8 & 5.6 & 1.20_{-7} & 12.1 & 7.1 & 0.8 & 4.79_{-6}\\
     73\,728 & 4.1 & 4.0 & 12.7 & 1.17_{-7} & 28.6 & 16.3 & 1.8 & 4.98_{-6}\\
    131\,072 & 7.2 & 7.2 & 23.6 & 1.20_{-7} & 53.3 & 29.8 & 3.3 & 5.07_{-6}\\
    294\,912 & 16.6 & 16.5 & 53.7 & 1.18_{-7} & 120.5 & 68.0 & 7.9 & 4.77_{-6}\\
    524\,288 & 29.4 & 29.3 & 97.3 & 1.10_{-7} & 215.2 & 121.7 & 13.4 & 4.87_{-6}\\
 1\,179\,648 & 65.8 & 65.7 & 218.1 & 1.00_{-7} & 472.2 & 264.3 & 28.7 & 5.53_{-6}\\
 2\,097\,152 & 116.7 & 117.2 & 393.7 & 1.20_{-7} & 875.9 & 487.4 & 53.5 & 4.76_{-6}
    \end{array}
  \end{equation*}

  \caption{Run-times and relative spectral errors for the multiplication
    of the single-layer matrix}
  \label{ta:slp}
\end{table}

Table~\ref{ta:slp} lists the results of applying the new algorithm
to multiply the single-layer matrix with itself.
The ``Induced'' columns refer to the compression of the induced
row and column cluster bases:
$t_\text{row}$ gives the time in seconds for the row basis,
$t_\text{col}$ the time for the column basis, and
$t_\text{mat}$ the time for forming the $\mathcal{H}^2$-matrix
using these bases and the full block tree $\ctprod$.
The value $\epsilon_2$ gives the relative spectral
error, estimated by twenty steps of the power iteration.
Here $1.12_{-7}$ is a short notation for $1.12\times 10^{-7}$.

The ``Final'' columns refer to the approximation of the product
in the coarser block tree $\ctIK$.
The block tree $\ctIK$ is constructed with the
same standard admissibility condition used for $\ctIJ$ and $\ctJK$.
$t_\text{row}$ gives the time in seconds for the construction of the
row basis for the coarse block tree,
$t_\text{col}$ the time for the column basis, and
$t_\text{mat}$ the time for forming the final $\mathcal{H}^2$-matrix
with the coarse block tree $\ctIK$.
The intermediate representation of the product is
set up during the construction of the row basis and re-used during
the construction of the column basis.
This explains why the column basis takes considerably less time.
The value $\epsilon_2$ is again the estimated
relative spectral error.

These experiments were carried out in a sequential
implementation of the algorithm on an AMD EPYC 7713 processor using
the shared-memory AOCL-BLIS library for linear algebra
subroutines.
Preliminary experiments indicate that
distributing the clusters among multiple processor cores can
significantly reduce the computing time of both phases of
the algorithm.

The time required to set up the cluster basis products $V_{X,s}^* W_{Y,s}$
and the total weights $Z_t$ for the first algorithm are not listed,
since these operations' run-times are negligible compared to the
new algorithms.

We can see that the relative errors remain well below the
prescribed bound of $10^{-4}$ and that the relative errors for the
first phase are even smaller.

%
% Table: Multiplication of the double-layer matrix
%
\begin{table}
  \begin{equation*}
    \begin{array}{r|rrrr|rrrr}
      & \multicolumn{4}{|c}{\text{Induced}} & \multicolumn{4}{|c}{\text{Final}}\\
    n & t_\text{row} & t_\text{col} & t_\text{mat} & \epsilon_2
      & t_\text{row} & t_\text{col} & t_\text{mat} & \epsilon_2\\
      \hline
      3\,072 & 0.2 & 0.2 & 0.4 & 9.13_{-7} & 0.9 & 0.4 & 0.1 & 1.34_{-5}\\
      6\,912 & 0.6 & 0.4 & 1.2 & 1.57_{-6} & 3.1 & 0.8 & 0.3 & 1.14_{-5}\\
     12\,288 & 1.2 & 0.7 & 2.2 & 1.89_{-6} & 5.4 & 2.9 & 0.4 & 7.46_{-6}\\
     27\,648 & 2.9 & 1.7 & 5.7 & 2.90_{-6} & 14.5 & 8.6 & 1.3 & 7.58_{-6}\\
     49\,152 & 5.3 & 3.0 & 10.1 & 3.43_{-6} & 23.9 & 13.1 & 1.9 & 7.21_{-6}\\
    110\,592 & 12.1 & 7.0 & 23.9 & 4.52_{-6} & 57.9 & 32.6 & 4.2 & 7.02_{-6}\\
    196\,608 & 21.4 & 12.3 & 41.9 & 4.95_{-6} & 97.5 & 49.5 & 6.1 & 6.63_{-6}\\
    442\,368 & 47.9 & 27.7 & 97.3 & 5.72_{-6} & 236.1 & 122.1 & 15.9 & 6.35_{-6}\\
    786\,432 & 86.5 & 47.4 & 169.0 & 6.15_{-6} & 388.9 & 187.7 & 49.3 & 6.76_{-6}\\
 1\,769\,472 & 188.6 & 108.8 & 390.6 & 6.75_{-6} & 926.2 & 454.2 & 72.2 & 7.79_{-6}\\
 3\,145\,728 & 338.7 & 190.7 & 673.0 & 7.12_{-6} & 1562.5 & 765.6 & 109.9 & 7.79_{-6}
    \end{array}
  \end{equation*}

  \caption{Run-times and relative spectral errors for the multiplication
    of the double-layer matrix}
  \label{ta:dlp}
\end{table}

Table~\ref{ta:dlp} shows the corresponding results for the double-layer
matrix on the surface of the cube.
We can again observe that the required accuracy is reliably provided.

%
% Figure: Run-time per degree of freedom
%
\begin{figure}
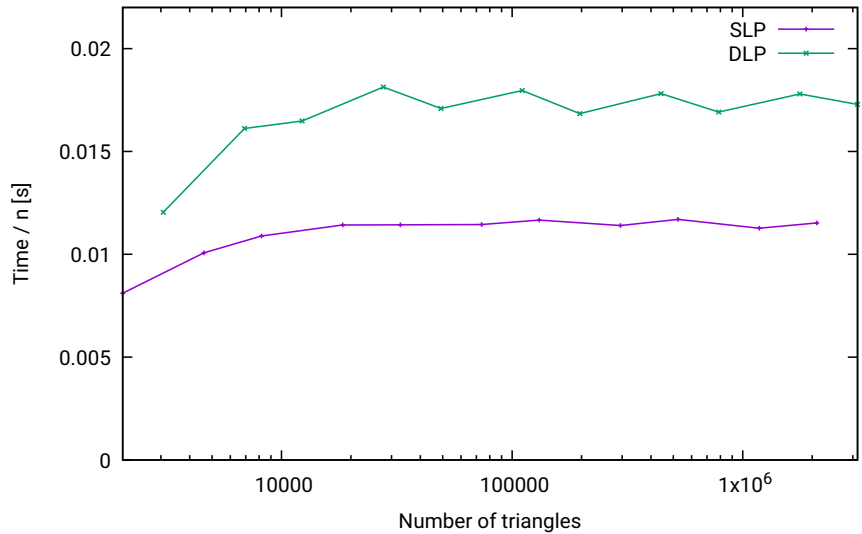

  \pgfdeclareimage[width=0.8\textwidth]{coarsened}{fi_coarsened}

  \begin{center}
    \pgfuseimage{coarsened}
  \end{center}

  \caption{Run-time for the adaptive $\mathcal{H}^2$-matrix multiplication
    per degree of freedom.}
  \label{fi:runtime}
\end{figure}

To see that the new algorithm indeed reaches the optimal linear complexity
with respect to the matrix dimension $n$, we show the time divided by $n$
in Figure~\ref{fi:runtime} for the single-layer potential
(SLP) on the unit sphere and the double-layer potential (DLP) on the
unit cube.
We can see that the runtime \emph{per degree of freedom} is bounded
uniformly both for the single-layer and the double-layer matrix, i.e.,
we observe $\mathcal{O}(n)$ complexity.

\smallskip

In conclusion, we have found an algorithm that can approximate the product
of two $\mathcal{H}^2$-matrices at a prescribed accuracy in linear, i.e.,
optimal complexity.
In the first phase, the algorithm approximates the induced cluster bases
for a refined intermediate block tree $\ctprod$.
In the second phase, an optimized representation for a prescribed block
tree $\ctIK$ is constructed.
Both phases can be performed with localized error control for all
submatrices.

The fact that the matrix multiplication can indeed be performed in
linear complexity leads us to hope that similar algorithms can be
developed for important operations like the Cholesky or LU factorization,
which would immediately give rise to efficient solvers for large
systems of linear equations.

% ============================================================
% References
% ============================================================

\bibliographystyle{plain}
\bibliography{hmatrix}

\end{document}